\documentclass[final]{siamart171218}

\usepackage{amsmath,amssymb,mathtools}
\usepackage{tikz}

\usepackage{enumitem}
\usepackage{tabularx}
\newcolumntype{Y}{>{\centering\arraybackslash}X}

\usepackage{booktabs,makecell}
\usepackage{array}
\usepackage{xcolor}
\usepackage{hyphenat} 

\usepackage{xparse}
\NewDocumentCommand{\son}{}{\mathrm{SO}(n)}

\usepackage{easyReview}

\usepackage[utf8]{inputenc} 
\usepackage[T1]{fontenc}

\newsiamthm{conjecture}{Conjecture}
\newsiamremark{remark}{Remark}
\newsiamremark{example}{Example}

\theoremstyle{nonumberplain}
\theoremheaderfont{\normalfont\scshape}
\theorembodyfont{\itshape}
\newtheorem{claim}{Claim}

\theoremsymbol{\ensuremath{\blacklozenge}}
\theoremheaderfont{\itshape}
\theorembodyfont{\normalfont}
\newtheorem{claimproof}{Proof of Claim}

\newcommand{\lie}{\mathrm{Lie}\,}
\newcommand{\tr}{\intercal}

\usepackage{bookmark}
\usepackage{cleveref}

\title{Combinatorics-Based Approaches to Controllability Characterization for Bilinear Systems}

\author{Gong Cheng\thanks{Department of Electrical and Systems
    Engineering, Washington University, St.\ Louis, MO 63130
    (\email{gong.cheng@wustl.edu}, \email{wei.zhang@wustl.edu}, \email{jsli@wustl.edu}). Questions, comments, or corrections regarding this document may be directed to Li.}
  \and Wei Zhang\footnotemark[1]
  \and Jr-Shin Li\footnotemark[1]}

\begin{document}
\maketitle


\tikzstyle{vertex}=[circle, draw, fill=black!50, inner sep=0pt,
minimum width=4pt]

\newdimen\R
\R=8mm


\begin{abstract}
  The control of bilinear systems has attracted considerable attention in the field of systems and control for decades, owing to their prevalence in diverse applications across science and engineering disciplines. Although much work has been conducted on analyzing controllability properties, the mostly used tool remains the Lie algebra rank condition. In this paper, we develop alternative approaches based on theory and techniques in combinatorics to study controllability of bilinear systems. The core idea of our methodology is to represent vector fields of a bilinear system by permutations or graphs, so that Lie brackets are represented by permutation multiplications or graph operations, respectively. Following these representations, we derive combinatorial characterization of controllability for bilinear systems, which consequently provides novel applications of symmetric group and graph theory to control theory. Moreover, the developed combinatorial approaches are compatible with Lie algebra decompositions, including the Cartan and non-intertwining decomposition. This compatibility enables the exploitation of representation theory for analyzing controllability, which allows us to characterize controllability properties of bilinear systems governed by semisimple and reductive Lie algebras.
\end{abstract}

\begin{keywords}
  Bilinear systems, Lie groups, graph theory, symmetric groups,
  representation theory, Cartan decomposition
\end{keywords}


\section{Introduction}

Bilinear systems, a class of nonlinear systems, emerge naturally as mathematical models to describe the dynamics of numerous processes in science and engineering. Prominent examples include the Bloch system governing the dynamics of spin-$\tfrac{1}{2}$ nuclei immersed in a magnetic field in quantum physics \cite{Glaser98,Li_PRA06,Li_PNAS11}, the compartmental model describing the movement of cells and molecules in biology \cite{Mohler78,Eisen79,Mohler80}, and the integrate-and-fire model characterizing the membrane potential of a neuron under synaptic inputs and injected current in neuroscience \cite{Dayan05,Gerstner02}. The prevalence of bilinear systems has been actively promoting the research in control theory and engineering concerning the analysis and manipulation of such systems for decades. The initial investigation into control problems involving bilinear systems traces back to the year of 1935, when the Greek mathematician Constantin Carath\'{e}odory studied optimal control of bilinear systems presented in terms of Pfaffian forms by using calculus of variations and partial differential equations \cite{Caratheodory35}. However, research in systematic analysis of fundamental properties of bilinear control systems was not prosperous until the early 1970s, when leading control theorists, such as Brockett, Jurdjevic, and Sussmann, developed geometric control theory for introducing techniques in Lie theory and differential geometry to classical control theory \cite{Brockett14,Brockett72,Jurdjevic72,Hirschorn75,Brockett76,Hermann77}. 

One of the most remarkable results in geometric control theory is the Lie algebra rank condition (LARC), which establishes an equivalence between controllability of control-affine systems defined on smooth manifolds and Lie algebras generated by the vector fields governing the system dynamics \cite{Brockett72,Isidori95,Jurdjevic96}. In our recent work, based on the LARC, we developed a necessary and sufficient controllability condition for bilinear systems by using techniques in symmetric group theory \cite{Zhang19}. In particular, we introduced a monoid structure on symmetric groups so that Lie bracket operations are compatible with monoid operations. This then resulted in a characterization of controllability in terms of elements in ``symmetric monoids'' for bilinear systems, which also offered an alternative to the LARC and further shed light on interpreting geometric control theory from an algebraic perspective.

In this paper, we propose a combinatorics-based framework to analyze controllability of bilinear systems defined on Lie groups by adopting techniques in symmetric group theory and graph theory. Specifically, the main idea is to  associate such systems with permutations or graphs, so that Lie bracket operations of the vector fields governing the system dynamics can be represented by permutation multiplications and edge operations on the graphs. This combinatorics approach immediately leads to the characterizations of controllability in terms of permutation cycles and graph connectivity. In particular, we identify the classes of bilinear systems, for which controllability has equivalent symmetric group and graph representations. A prominent example is the system defined on $\son$, the special orthogonal group, for which we reveal a correspondence between permutation cycles in the symmetric group and trees in the graph associated with these systems. It is worth noting that, different from our previous work on the symmetric group method \cite{Zhang19}, the correspondence between Lie bracket operations and permutation multiplications established in this paper do not require any monoid structure on symmetric groups. On the other hand, the application of graph theory in the developed combinatorics-based framework offers a distinct viewpoint to the field of control theory. Specifically, in the existing literature, graphs are naturally used in the context of networked and multi-agent systems, e.g., for describing the coupling topology and deriving structural controllability conditions \cite{Mou2016,Qin2016,Tsopelakos2019}, while, in this work, we establish a non-trivial relationship between graph connectivity and controllability for a single bilinear system.

Moreover, a great advantage of the developed framework is its compatibility with various Lie algebra decomposition techniques in representation theory. In particular, we illustrate the application of these methods to systems of which the underlying Lie algebras are semisimple or reductive, while in these cases, the correspondence between Lie bracket operations and permutation multiplications as well as graph operations is elusive due to their complicated algebraic structures. In this work, we exploit the Cartan and non-intertwining decompositions to decompose the system Lie algebras into simple components, so that the combinatorics-based controllability analysis is equivalently carried over to these components. 

This paper is organized as follows. In Section~\ref{sec:prelim}, we provide the preliminaries relevant to our developments, including the LARC for systems on Lie groups and a brief review of the Lie algebra $\mathfrak{so}(n)$. In Section~\ref{sec:com.framework}, we establish the symmetric group and graph-theoretic methods based upon the study of bilinear systems on $\son$. In Section~\ref{sec:non-standard}, we introduce the notions and tools of Cartan and non-intertwining decompositions for decomposing the system Lie algebras into simpler components, which enables and facilitates the generalization of the combinatorics-based framework to broader classes of bilinear systems. A brief review of the basics of symmetric groups and Lie algebra decompositions can be found in the appendices.


\section{Preliminaries} \label{sec:prelim}

To prepare for our development of the combinatorial controllability conditions, in this section, we briefly review the Lie algebra $\mathfrak{so}(n)$ and the LARC for right-invariant bilinear systems. Meanwhile, we introduce the notations we use throughout this paper.

\subsection{The Lie Algebra Rank Condition}
The LARC has been the most recognizable tool, if not unique, for analyzing controllability of bilinear systems since the 1970s. It establishes a connection between controllability and the Lie algebra generated by the vector fields governing the system dynamics. In this paper, we primarily focus on the bilinear system evolving on a compact and connected Lie group of the form,
\begin{equation}
  \label{eq:bilinear.system}
  \dot{X}(t)=B_0X(t)+\Bigl(\sum_{i=1}^mu_i(t)B_i\Bigr)X(t),\quad{}X(0)=I,
\end{equation}
where $X(t)\in G$ is the state on a compact and connected Lie group $G$, $I$ is the identity element of $G$, $B_i$ are elements in the Lie algebra $\mathfrak{g}$ of $G$, and $u_i(t)\in\mathbb{R}$ are piecewise constant control inputs. For any subset $\Gamma\subseteq\mathfrak{g}$, we use $\lie(\Gamma)$ to denote the Lie subalgebra generated by $\Gamma$, i.e., the smallest vector subspace of $\mathfrak{g}$ containing $\Gamma$ that is closed under the Lie bracket defined by $[C,D]:=CD-DC$ for $C,D\in\mathfrak{g}$. With these notations, the LARC for the system in \cref{eq:bilinear.system} can be stated as follows.

\begin{theorem}[LARC]
  \label{thm:LARC}
  The system in \cref{eq:bilinear.system} is \emph{controllable} on   $G$ if and only if $\lie(\Gamma)=\mathfrak{g}$, where $\Gamma=\{B_0, B_1, \ldots, B_m\}$.
\end{theorem}

\begin{proof}
  See \cite{Brockett72}.
\end{proof}

\subsection{Basics of the Lie Algebra \texorpdfstring{{$\mathfrak{so}(n)$}}{so(n)}}

The Lie algebra $\mathfrak{so}(n)$ is a vector space of dimension $n(n-1)/2$, which consists of all $n$-by-$n$ real skew-symmetric matrices. In particular, if we use $\Omega_{ij}$ to denote the skew-symmetric matrix with $1$ in the $(i,j)$-th entry, ${-1}$ in the $(j,i)$-th entry, and $0$ elsewhere, then the set $\mathcal{B} =\{\Omega_{ij}\in\mathbb{R}^{n\times{}n}: 1\leqslant{}i<j\leqslant{}n\}$ forms a basis of $\mathfrak{so}(n)$, which we refer to as the \emph{standard basis} of
$\mathfrak{so}(n)$. The following lemma then reveals the Lie bracket relations among elements in $\mathcal{B}$.

\begin{lemma}
  \label{lem:son}
  The Lie bracket of $\Omega_{ij}$ and $\Omega_{kl}$ satisfies the
  relation
  $[\Omega_{ij},\Omega_{kl}]=\delta_{jk}\Omega_{il}+\delta_{il}\Omega_{jk}+\delta_{jl}\Omega_{ki}+\delta_{ik}\Omega_{lj}$,
  where $\delta$ is the Kronecker delta function defined by
  \[
    \delta_{mn}=
    \begin{cases}
      1, & \text{if }m=n; \\
      0, & \text{otherwise}.
    \end{cases}
  \]
\end{lemma}

\begin{proof}
  The proof follows directly from computations.
\end{proof}

The relations in \cref{lem:son} can also be equivalently expressed as
$[\Omega_{ij},\Omega_{kl}]\neq{}0$ if and only if $i=k$, $i=l$, $j=k$,
or $j=l$. This algebraic structure facilitates controllability
characterization of the bilinear system governed by the vector fields
represented in the standard basis $\mathcal{B}$, which is the main
focus of the next section.

\section{Combinatorics-Based Controllability Analysis for Bilinear Systems} 
\label{sec:com.framework}

In this section, we introduce a combinatorics-based framework to characterize controllability of bilinear systems. Within this framework, we adopt tools in two subfields of combinatorics - the symmetric group theory and graph theory, and build connections of Lie brackets of vector fields to permutation multiplications in symmetric groups and operations on graph edges, respectively. Such connections enable us to characterize controllability in terms of permutation cycles and graph connectivity. Here, we will investigate bilinear systems defined on $\son$, given by 
\begin{equation}
  \label{eq:system_SOn}
  \dot{X}(t)=\Omega_{i_0j_0}X+\Bigl(\sum_{k=1}^mu_i(t)\Omega_{i_k
    j_k}\Bigr)X, \quad \Omega_{i_kj_k}\in\mathcal{B}, \quad X(0)=I.
\end{equation}
as building blocks to establish this framework.
Furthermore, we will show that owing to the special algebraic structure of $\mathfrak{so}(n)$ presented in \cref{lem:son}, the symmetric group and the graph-theoretic approach, when applied to \eqref{eq:system_SOn}, give an equivalent characterization of controllability through an interconnection between symmetric groups and graphs.

\subsection{The Symmetric Group Method for Controllability Analysis} 
\label{sec:Sn}

In this section, we introduce the symmetric group method for analyzing controllability of the system in \eqref{eq:system_SOn}. In this approach, a subset of vector fields in $\mathcal{B}$ is represented using a permutation in $S_n$, the symmetric group of $n$ letters. Through this representation, we connect the Lie brackets of vector fields to permutation multiplications, so that controllability is determined by the length of permutation cycles. For a brief summary of symmetric groups and permutations, see \cref{appd:Sn}.

\subsubsection{Mapping Lie Brackets to Permutations}

To establish a relation from Lie brackets to permutation multiplications, we first define a \emph{relation} between subsets of $\mathcal{B}$ and permutations in $S_n$ by
\begin{equation}
  \label{eq:iota}
  \iota:\mathcal{P}(\mathcal{B})\rightarrow S_n,\quad
  \iota(\{\Omega_{i_0j_0},\Omega_{i_1j_1},\dots,\Omega_{i_mj_m}\})=
  (i_0j_0)(i_1j_1)\cdots(i_mj_m).
\end{equation}
Because every permutation can be decomposed into a product of transpositions ($2$-cycles), the relation $\iota$ is surjective so that every subset of $\mathcal{B}$ admits a permutation
representation.

To see how Lie bracket operations are related to permutation multiplications by $\iota$, we illustrate the idea using two elements $\Omega_{ij},\Omega_{kl}\in\mathcal{B}$. On the Lie algebra level, if
$[\Omega_{ij},\Omega_{kl}]\neq0$, then \cref{lem:son} implies that $\{i,j\}$ and $\{k,l\}$ have a common index. Without loss of generality, we may assume $j=k$ and $i\neq l$, so that $[\Omega_{ij},\Omega_{jl}]=\Omega_{il}$. Meanwhile, on the symmetric group level, we have $\iota(\{\Omega_{ij},\Omega_{jl}\})=(ij)(jl)=(ijl)$, so the permutation multiplication increases the cycle length by $1$, from the $2$-cycle factors $(ij)$ and $(jk)$ to a $3$-cycle $(ijk)$. However, if $[\Omega_{ij},\Omega_{kl}]=0$, then $\{i,j\}\cap\{k,l\}=\emptyset$, and $\iota(\{\Omega_{ij},\Omega_{kl}\})=(ij)(kl)$ is a product of two disjoint cycles. The phenomenon that elements in $\mathcal{B}$ with non-vanishing Lie brackets relating to a cycle with increased length extends inductively to larger subsets of $\mathcal{B}$. To be more specific, if $\Gamma\subset\mathcal{B}$ contains $m$ elements such that the iterated Lie brackets of them are non-vanishing, then $\iota(\Gamma)$ is an $(m+1)$-cycle. This observation immediately motivates the use of cycle length to examine controllability of systems on $\son$ in \cref{eq:system_SOn}. Before we state and prove our main theorem, let us first illustrate the symmetric group method by two examples.

\begin{example}
  \label{ex:so(5)}
  Consider a system evolving on $\mathrm{SO}(5)$, given by
  \begin{equation}
    \label{eq:ex_so(5)}
    \dot{X}(t)=\Bigl(\sum_{i=1}^4u_i(t)\Omega_{i,i+1}\Bigr)X(t),\quad X(0)=I,
  \end{equation}
  and let $\Gamma=\{\Omega_{i,i+1}: i=1,\ldots,4\}$ denote the set of
  control vector fields. The correspondence between Lie brackets in
  $\Gamma$ and permutation multiplications in $S_5$ follows
  \begin{equation}
    \label{eq:so5_S5}
    \begin{aligned}
      [\Omega_{12},\Omega_{23}] &=\Omega_{13} & &\leftrightarrow &
      (12)(23) &=(123), \\
      [\Omega_{23},\Omega_{34}] &=\Omega_{24} & &\leftrightarrow &
      (23)(34) &=(234), \\
      [\Omega_{34},\Omega_{45}] &=\Omega_{35} & &\leftrightarrow &
      (34)(45) &=(345), \\
      [\Omega_{12},[\Omega_{23},\Omega_{34}]] &=\Omega_{14} &
      &\leftrightarrow &
      (12)(234) &=(1234), \\
      [\Omega_{23},[\Omega_{34},\Omega_{45}]] &=\Omega_{25} &
      &\leftrightarrow &
      (23)(345) &=(2345), \\
      [\Omega_{12},[\Omega_{23},[\Omega_{34},\Omega_{45}]]]
      &=\Omega_{15} & &\leftrightarrow & (12)(2345) &=(12345).
    \end{aligned}
  \end{equation}
  Note that successively Lie bracketing elements in $\Gamma$ results
  in $\Omega_{13}$, $\Omega_{14}$, $\Omega_{15}$, $\Omega_{24}$,
  $\Omega_{25}$, and $\Omega_{35}$, together with the $4$ elements in
  $\Gamma$, we have $10$ linearly independent vector fields. Because
  $\mathfrak{so}(5)$ is a $10$-dimensional Lie algebra, we conclude
  $\lie(\Gamma)=\mathfrak{so}(5)$, which implies that the system in
  \cref{eq:ex_so(5)} is controllable on ${\rm SO}(5)$ by the LARC. On
  the other hand, \cref{eq:so5_S5} also shows $\iota(\Gamma)=(12345)$,
  a cycle of maximum length in $S_5$.  This suggests that
  controllability of systems on $\son$ can be characterized by cycles
  of \emph{maximum} length in the corresponding symmetric group.
\end{example}

\begin{example} 
  \label{ex:so(5)_2}
  Consider another system evolving on ${\rm SO}(5)$ driven by three
  controls, given by
  \begin{equation}
    \label{eq:ex_so(5)_2}
    \dot{X}(t)=\bigl(u_1(t)\Omega_{12}+u_2(t)\Omega_{23}+u_3(t)\Omega_{45}\bigr)X(t),
    \quad X(0)=I.
  \end{equation}
  In this case, the single Lie brackets,
  \[
    \begin{aligned}
      [\Omega_{12},\Omega_{23}] &=\Omega_{13} & &\leftrightarrow & &(12)(23)=(123),\\
      [\Omega_{12},\Omega_{45}] &=0 & &\leftrightarrow & &(12)(45),\\
      [\Omega_{23},\Omega_{45}] &=0 & &\leftrightarrow & &(23)(45),
    \end{aligned}
  \]
  and the double Lie brackets,
  \[
    \begin{aligned}
      [\Omega_{13},\Omega_{12}]
      &=[[\Omega_{12},\Omega_{23}],\Omega_{12}]=\Omega_{23} & &\leftrightarrow
      & (12)(23)(12) &=(13), \\
      [\Omega_{23},\Omega_{13}]
      &=[\Omega_{23},[\Omega_{12},\Omega_{23}]]=\Omega_{12} & &\leftrightarrow 
      & (23)(12)(23) &=(13), \\
      [\Omega_{13},\Omega_{45}]
      &=[[\Omega_{12},\Omega_{23}],\Omega_{45}]=0 & &\leftrightarrow
      & (12)(23)(45) &=(123)(45),
    \end{aligned}
  \]
  result in a Lie subalgebra of dimension $4$. Therefore, this system
  is \emph{not} controllable on ${\rm SO}(5)$. On the other hand, for
  $\Gamma=\{\Omega_{12}, \Omega_{23}, \Omega_{45}\}$, the computations
  above also show $\iota(\Gamma)=(123)(45)$, which is \emph{not} a
  single cycle of maximum length in $S_5$.
\end{example}

\Cref{ex:so(5),ex:so(5)_2} together verify the observation that cycles
with the maximum length characterize controllability of bilinear
systems on $\son$, which we will prove in the next section.

\begin{remark}
 \label{rmk:iota.not.a.map}
 Note that the relation $\iota$ introduced in \cref{eq:iota} is
 \emph{not} a well-defined function, because, for a given
 $\Gamma\subseteq B$, $\iota(\Gamma)$ depends on the ordering of the
 elements in $\Gamma$. If, say, $\Gamma=\{\Omega_{12}, \Omega_{14},
 \Omega_{23}, \Omega_{24}, \Omega_{34}\}$, then different element
 orderings,
 \[
   \begin{aligned}
     &\{\Omega_{12}, \Omega_{14}, \Omega_{23}, \Omega_{24},
     \Omega_{34}\} & &\leftrightarrow & (12)(14)(23)(24)(34) &=(14) \\
     &\{\Omega_{14}, \Omega_{12}, \Omega_{24}, \Omega_{23},
     \Omega_{34}\} & &\leftrightarrow & (14)(12)(24)(23)(34) &=(1234)
   \end{aligned}
 \]
 could result in \emph{different} permutations. 
 Nevertheless, we can verify that for any
 $\Gamma\subseteq\mathcal{B}$, there always exists a subset
 $\Sigma\subseteq\Gamma$ such that $\iota$ relates $\Sigma$ to
 permutations with the same (maximal) orbits, albeit different
 orderings of the elements in $\Sigma$. For example, for the subset
 $\Sigma=\{\Omega_{12}, \Omega_{23}, \Omega_{34}\}$ of $\Gamma$,
 $\iota(\Sigma)$ is always a $4$-cycle with its orbit being
 $\{1,2,3,4\}$, regardless of its element orderings. The existence of
 such a subset will be clear once we develop a graph visualization of
 the permutations in Section~\ref{sec:graph}.
\end{remark}

\subsubsection{Controllability Characterization in Terms of Permutation Cycles}
Leveraging the technique of mapping Lie brackets to permutations developed in the previous section, we are able to characterize controllability of systems on $\son$ in terms of permutation cycles as shown in the following theorem.

\begin{theorem}
  \label{thm:SOn_Sn}
  The control system defined on $\son$ of the form
  \begin{equation}
    \label{eq:son}
    \dot{X}(t)=\Bigl(\Omega_{i_0j_0}+\sum_{k=1}^mu_k(t)\Omega_{i_kj_k}\Bigr)X(t), \quad X(0)=I,
  \end{equation}
  (same system as (3.1)) where $\Gamma:=\{\Omega_{i_k j_k}\}\subseteq\mathcal{B}$ for $k=0,\dots,m$, is controllable if and only if there is a subset $\Sigma\subseteq\Gamma$ such that $\iota(\Sigma)$ is an $n$-cycle, where $\iota$ is the relation defined in \cref{eq:iota}.
\end{theorem}

\begin{proof} 
  By the LARC, the system in \cref{eq:son} is controllable on $\son$ if and only if $\lie(\Gamma)=\mathfrak{so}(n)$. Therefore, it is equivalent to showing that $\lie(\Sigma)=\mathfrak{so}(n)$ if and only if $\iota(\Sigma)$ is an $n$-cycle for some $\Sigma\subseteq\Gamma$.
	
  (Sufficiency): Suppose there exists a subset $\Sigma\subseteq\Gamma$
  such that $\iota(\Sigma)$ is an $n$-cycle. Because an $n$-cycle can
  be decomposed into a product of at least $n-1$ transpositions, this
  implies $m\geqslant n-1$. Hence, it suffices to assume that the
  cardinality of $\Sigma$ is $n-1$, and, without loss of generality,
  let $\Sigma=\{\Omega_{i_1j_1},\dots,\Omega_{i_{n-1}j_{n-1}}\}$,
  Because $\iota(\Sigma)$ is an $n$-cycle, it follows that the index
  set $\{i_1,j_1,\dots,i_{n-1},j_{n-1}\}=\{1,\ldots,n\}$.  Note that
  the set $\{i_1,j_1,\dots,i_{n-1},j_{n-1}\}$ contains repeated
  elements. Next, we prove the sufficiency by induction.

  When $n=3$, suppose there exists a subset
  $\Sigma=\{\Omega_{ij},\Omega_{kl}\}\subset\Gamma$ and that
  $\iota(\Sigma)=(ij)(kl)$ is a 3-cycle, so we must have one of the
  following: $i=k$, $j=k$, $i=l$, or $j=l$. Consequently,
  $[\Omega_{ij},\Omega_{kl}]\in \mathcal{B}\backslash\Sigma$, so
  $\{\Omega_{ij},\Omega_{kl},[\Omega_{ij},\Omega_{kl}]\}$ spans
  $\mathfrak{so}(3)$. Therefore, the system in \cref{eq:son} is
  controllable on $\mathrm{SO}(3)$.

  Now let us assume that for $n\geqslant 4$, a system defined on
  ${\rm SO}(n-1)$ in the form of \cref{eq:son} is controllable if
  there is $\Sigma\subseteq\Gamma$ such that $\iota(\Sigma)$ is an
  $(n-1)$-cycle.  Let $\Sigma\subseteq\Gamma$ be a set of $n-1$
  elements such that
  $\iota(\Sigma)=(i_{n-1}j_{n-1})(i_{n-2}j_{n-2})\cdots(i_1j_1)$ is a
  cycle of length $n$, then for every integer
  $1 \leqslant{}k \leqslant{}n-1$, there exists some
  $1 \leqslant{}l \leqslant{}n-1$ such that
  $\{i_k,j_k\}\cap\{i_l,j_l\}\neq\emptyset$. Consequently, there are
  $n-2$ transpositions of the form $(i_kj_k)$, $k=1,\dots,n-1$, such
  that their product is a cycle of length $n-1$. Without loss of
  generality, we may assume that
  $\iota(\Sigma\backslash\{\Omega_{i_{n-1}j_{n-1}}\})
  =(i_{n-2}j_{n-2})\cdots(i_1j_1)$ is a $(n-1)$-cycle with the
  nontrivial orbit
  $\{i_1,j_1,\dots,i_{n-2},j_{n-2}\} =\{1,\dots,n-1\}$. By the
  induction hypothesis, the system in \cref{eq:son} is controllable on
  $\mathrm{SO}(n-1)\subset\son$. Equivalently, any $\Omega_{ij}\in \mathcal{B}$ such that
  $1 \leqslant{} i<j\leqslant n-1$ can be generated by iterated Lie
  brackets of the elements in
  $\Sigma\backslash\{\Omega_{i_{n-1}j_{n-1}}\}$. Because
  $\iota(\Sigma)=(i_{n-1}j_{n-1})\iota(\Sigma\backslash\{\Omega_{i_{n-1}j_{n-1}}\})$
  is a $n$-cycle, we must have $i_{n-1}\in\{1,\dots,n-1\}$ and
  $j_{n-1}=n$. Therefore, $\Omega_{kn}$ can be generated by the Lie
  brackets $[\Omega_{ki_{n-1}},\Omega_{i_{n-1}j_{n-1}}]$ for any
  $k=1,\dots,n-1$. As a result, the system in \cref{eq:son} is
  controllable on $\son$.

  (Necessity): Because the system in \cref{eq:son} is controllable,
  $\lie(\Gamma)=\mathfrak{so}(n)$. Then, there exists a subset
  $\Sigma$ of $\Gamma$ such that $\lie(\Sigma)=\mathfrak{so}(n)$ and
  $\Sigma$ contains no \emph{redundant elements}, i.e., the elements
  that can be generated by Lie brackets of the other elements in
  $\Sigma$. Without loss of generality, we assume
  $\Sigma=\{\Omega_{i_1j_1},\dots,\Omega_{i_lj_l}\}$, where
  $l\leqslant m$. By \cref{lem:son}, for any
  $\Omega_{ab},\Omega_{cd}\in\Sigma$, if
  $[\Omega_{ab},\Omega_{cd}]\neq 0$, then there must exist a bridging
  index, i.e., we must have one of the following cases: $a=c$, $a=d$,
  $b=c$, or $b=d$. This, together with
  $\lie(\Sigma)=\mathfrak{so}(n)$, implies that the index set $J$ of
  $\Sigma$ is $J=\{i_1,j_1,\dots,i_l,j_l\}=\{1,\dots,n\}$, and that
  for any $\Omega_{i_kj_k}\in\Sigma$, there exists some
  $\Omega_{i_sj_s}\in \Sigma$ with $s\neq k$ such that
  $\{i_k,j_k\}\cap\{i_s,j_s\}\neq\emptyset$. Moreover, because
  $\Sigma$ contains no redundant elements,
  $\iota(\Sigma)=\iota(\Omega_{i_lj_l})\cdots\iota(\Omega_{i_1j_1})$
  is a cycle whose orbit contains every element in $\{1,\dots,n\}$,
  namely, it is a cycle of length $n$. In addition, the cardinality of
  $\Sigma$ is $n-1$.
\end{proof}

\begin{remark}
  \label{rem:number_of_control}
  Following the above proof, it requires at least $n-1$ controls for the system on $\son$ in \cref{eq:son} to be fully controllable and, on the other hand, for $\iota(\Sigma)$, $\Sigma\subseteq\Gamma$, to reach a cycle of length $n$.
\end{remark}

Similar to the case in \cref{thm:SOn_Sn} for controllable systems, the
controllable submanifold for an uncontrollable system also depends on
the permutation related to a subset of $\Gamma$. To be more specific,
the cycle decomposition of such a permutation determines the
involutive distribution of the submanifold.

\begin{corollary}
  \label{cor:Sn_submanifold}
  Given a system evolving on $\son$ in the form of \cref{eq:system_SOn}, let $\Xi$ be a minimal subset of $\Gamma$, such that $\lie(\Xi)=\lie(\Gamma)$. If $\iota(\Xi)=\sigma_1\cdot\sigma_2\cdots\sigma_l$ so that each $\sigma_k$, $1 \leqslant{} k \leqslant{}l$, are pairwise disjoint
  cycles with the nontrivial orbits $\mathcal{O}_k$, then the controllable submanifold of the system is the Lie subgroup of $\son$ with the Lie algebra
  $\lie(\Gamma)=\bigoplus_{k=1}^l{\rm
    span}\,\{\Omega_{ij}:i,j\in\mathcal{O}_k\}$. Conversely, if
  $\lie(\Gamma)=\bigoplus_{k=1}^l{\rm
    span}\,\{\Omega_{ij}:i,j\in\mathcal{O}_k\}$ for some
  $\mathcal{O}_k\subset\{1,2,\dots,n\}$, then
  $\iota(\Xi)=\sigma_1\cdot\sigma_2\cdots\sigma_l$ and $\sigma_k$ are
  that disjoint cycles with nontrivial orbits $\mathcal{O}_k$.
\end{corollary}

\begin{proof}
  Let $\Xi$ be a minimal subset of $\Gamma$ such that
  $\lie(\Xi)=\lie(\Gamma)$ and $\Xi$ does not contain redundant
  elements. First, let $\sigma=\iota(\Xi)\in S_n$ be a cycle with
  nontrivial orbit $\mathcal{O}$, then \cref{thm:SOn_Sn} implies
  $\lie(\Xi)=\mathrm{span}\,\{\Omega_{ij}:i,j\in\mathcal{O},i<j\}$. Next,
  if $\sigma=\sigma_1\cdots\sigma_l$ is a permutation as a product of disjoint
  cycles $\sigma_1,\dots,\sigma_l$ with $l\geqslant 2$, then there exists a
  partition $\{\Xi_1,\dots,\Xi_l\}$ of $\Xi$ such that
  $\iota(\Xi_k)=\sigma_k$ for each $k=1,\dots,l$. Let $\mathcal{O}_k$
  denotes the nontrivial orbit of $\sigma_k$ for each $k=1,\dots,l$, then
  $\lie(\Xi_k)=\{\Omega_{ij}:i,j\in\mathcal{O}_k,i<j\}$ and the sets
  $\mathcal{O}_1,\dots,\mathcal{O}_l$ are pairwise disjoint subsets of
  $\{1,\dots,n\}$. Hence, $\lie(\Xi_i)\cap\lie(\Xi_j)=\{0\}$ holds for
  all $i\neq j$, and consequently, we have
  $\lie(\Xi)=\lie(\Xi_1)\oplus\cdots\oplus\lie\mathcal(\Xi_l)$, where
  $\oplus$ denotes the direct sum of vector spaces. By the Frobenius
  Theorem \cite{warner.lie.groups}, $\lie(\Xi)$ is completely
  integrable, and that the set of all its maximal integral manifolds
  forms a foliation $\mathcal{F}$ of $\son$. Since the initial
  condition of the system in \cref{eq:son} is the identity matrix $I$,
  the leaf of $\mathcal{F}$ passing through $I$ is the controllable
  submanifold of the system in \cref{eq:son}. The converse is obvious
  following a very similar argument.
\end{proof}

According to \cref{thm:SOn_Sn,cor:Sn_submanifold}, mapping the control
vector fields in $\Gamma$ to permutations provides not only an
alternative approach to effectively examine controllability of systems
defined on $\son$, but also a systematic procedure to characterize the
controllable submanifold when the system is not fully
controllable. Let us now revisit a previous example and see how
permutations help determine system controllability.

\begin{example}[Controllable Submanifold]
  \label{ex:controllable_submanifold}
  Recall \cref{ex:so(5)_2}, where the system in \cref{eq:ex_so(5)_2} is not controllable and there exist no subsets of $\Gamma=\{\Omega_{12},\Omega_{23},\Omega_{45}\}$ such that $\iota(\Gamma)$ is a $5$-cycle. In addition, the controllable submanifold is the integral manifold of the involutive distribution $\Delta=\lie\{\Omega_{12}X, \Omega_{23}X, \Omega_{13}X, \Omega_{45}X\}=\mathrm{span}\,\{\Omega_{ij}X: i,j\in\{1,2,3\}\text{ or }i,j\in\{4,5\}\}$, which can be identified by the nontrivial orbits of $\iota(\Gamma)=(1,2,3)(4,5)$. On the other hand, for each $X\in\mathrm{SO}(5)$, the complement $\Delta_X^{\perp}={\rm span}\,\{\Omega_{ij}X:i=1,2,3, j=4,5\}$ of the distribution evaluated at $X$ contains the bridging elements required for full controllability of this system.
\end{example}

\subsection{The Graph-Theoretic Method for Controllability Analysis} 
\label{sec:graph}

Graphs appear naturally in the research of networked systems, especially in modeling multi-agent systems and analyzing structural controllability \cite{Mou2016,Qin2016,Tsopelakos2019}. However,
most graph-theoretic methods were dedicated to studying networked control systems in existing literature and were not invented and applied for understanding fundamental properties of a single bilinear system. Here, we use graphs to represent the structure of Lie algebras and then characterize controllability of bilinear systems by graph connectivity. In contrast to the symmetric group method presented in Section \ref{sec:Sn}, this graph-theoretic method establishes a correspondence between Lie bracket operations of vector fields and operations on the edges of graphs. 

\subsubsection{Mapping Lie Brackets to Graphs} 
\label{sec:Lie_graph}
A graph $G$, conventionally denoted by a 2-tuple, $G=(V,E)$, consists of a vertex set $V$ and an edge set $E$. For the purpose of analyzing controllability of the system on $\son$, we are particularly interested in simple graphs, i.e., undirected graphs with no loops or multiple
edges, of $n$ vertices. Here, we denote the collection of such graphs $\mathcal{G}$. Without loss of generality, we further assume that every graph in $\mathcal{G}$ has the same vertex set
$V=\{v_1,\dots,v_n\}$. Following these notations, we define a map
\begin{equation}
  \label{eq:graph-map}
  \tau: \mathcal{P}(\mathcal{B})\rightarrow\mathcal{G}\quad{}\text{by}\quad{}
  \tau(\Gamma)=(V,E_{\Gamma}):=G_\Gamma,
\end{equation}
where $\mathcal{P}(\mathcal{B})$ denotes the power set of
$\mathcal{B}$, i.e., the set consisting of all subsets of
$\mathcal{B}$ and $E_\Gamma=\{v_iv_j:\Omega_{ij}\in\Gamma\}$.  Some
basic properties of $\tau$ are summarized in the following
proposition.

\begin{proposition}[Properties of $\tau$]
  \label{prop:property_tau}
  \mbox{} 
  \begin{enumerate}[font=\normalfont,label={(\roman*)}]%
    \item The map $\tau$ defined in \cref{eq:graph-map} is bijective.
    \item For any $\Gamma\subseteq\mathcal{B}$, $|\Gamma|=|E_\Gamma|$
      holds, where $|\cdot|$ denote the cardinality of a set.
    \item Let $K_n$ denote the complete graph of $n$ vertices, i.e.,
      the graph whose vertices are pairwise adjacent, then
      $\tau(\mathcal{B})=K_n$.
  \end{enumerate}
\end{proposition}

\begin{proof}
  Note that (i) and (ii) directly follow from the definition of
  $\tau$. For (iii), the edge set of $\tau(\mathcal{B})$ satisfies
  $E_{\mathcal{B}} =\{v_iv_j:\Omega_{ij}\in\mathcal{B}\}
  =\{v_iv_j:1\leqslant{}i<j\leqslant{}n\} =\{v_iv_j:i,j=1,\dots,n\}$,
  and hence we conclude $\tau(\mathcal{B})=K_n$.
\end{proof}

The property (i) in \cref{prop:property_tau} reveals a one-to-one
correspondence between the subsets of $\mathcal{B}$ and the graphs in
$\mathcal{G}$, which enables the representation of Lie bracket
operations by graph operations as follows.

Algebraically, for any $\Omega_{ij},\Omega_{jk}\in\mathcal{B}$,
\cref{lem:son} implies
$[\Omega_{ij},\Omega_{jk}] =\Omega_{ik}\neq{}0$, so that
$\lie\{\Omega_{ij},\Omega_{jk}\}
={\mathrm{span}}\{\Omega_{ij},\Omega_{jk},\Omega_{ik}\}$. Graphically,
by the definition of $\tau$, the two edges $\tau(\Omega_{ij})=v_iv_j$
and $\tau(\Omega_{jk})=v_jv_k$ share a common vertex $v_j$, and the
edge $\tau([\Omega_{ij},\Omega_{jk}])=\tau(\Omega_{ik})=v_iv_k$
intersects with $\tau(\Omega_{ij})$ and $\tau(\Omega_{jk})$ at
endpoints $v_i$ and $v_k$, respectively. Therefore, the three edges
$\tau(\Omega_{ij})$, $\tau(\Omega_{jk})$, and
$\tau([\Omega_{ij},\Omega_{jk}])$ form a triangle, or equivalently,
$\tau(\{\Omega_{ij},\Omega_{jk},[\Omega_{ij},\Omega_{jk}]\})
=\{v_iv_j,v_jv_k,v_iv_k\}=K_3$. This observation, as summarized in the
following lemma, reveals the relationship between first-order Lie
brackets and graph operations for three standard basis elements of
$\mathfrak{so}(n)$, which lays the foundation for the graph-theoretic
controllability analysis of bilinear systems.

\begin{lemma}
  \label{lem:lie-graph}
  If $\Omega_{ij},\Omega_{kl}\in\mathcal{B}$ satisfy
  $[\Omega_{ij},\Omega_{kl}]\neq0$, then
  \begin{enumerate}[font=\normalfont,label={(\roman*)}]%
    \item the two edges $\tau(\Omega_{ij})$ and $\tau(\Omega_{kl})$ are
      \emph{incident} (i.e., they share a common vertex);
    \item the three edges $\tau(\Omega_{ij})$, $\tau(\Omega_{kj})$,
      and $\tau([\Omega_{ij},\Omega_{kl}])$ form a triangle.
  \end{enumerate}
\end{lemma}

To graphically characterize higher-order Lie brackets among arbitrary
collections of standard basis elements of $\mathfrak{so}(n)$, we
introduce the notion of triangular closure for graphs, which
generalizes the action of ``forming triangles'' in
\cref{lem:lie-graph}.

\begin{definition}[Triangular Closure]
  \label{def:tri.closure}
  Let $G=(V,E)$ be a graph, and $\{G^m=(V,E^m):m=0,1,\dots\}$ be an
  ascending chain of graphs, i.e., $G^m\subseteq G^{m+1}$ for any
  $m=0,1,\dots$, satisfying
  \begin{enumerate}[font=\normalfont,label={(\roman*)}]%
    \item $G^0=G$, i.e., $E^0=E$.
    \item For any $m\geqslant{}0$, $v_iv_j\in{}E^{m+1}$ if and only if
      $v_iv_j\in{}E^m$ or there exists some vertex $v_k\in{}V$ such
      that $v_iv_k,v_kv_j\in{}E^m$.
  \end{enumerate} 
  Then the union of all $G^m$, denoted
  $\bar{G}=\bigcup_{m=1}^\infty G^m$, or equivalently,
  $\bar{G}=(V,\bar{E})=(V,\bigcup_{m=1}^{\infty}E^m)$, is called the
  \emph{triangular closure} of $G$. Moreover, a graph $G$ is called
  \emph{triangularly closed} if $G=\bar{G}$.
\end{definition}

Note that for a finite graph $G$, i.e., $G$ has finitely many vertices
and edges, the ascending chain of graphs
$G =G^0\subseteq{} G^1\subseteq\cdots$ in \cref{def:tri.closure}
stabilizes in finite steps, that is, there exists a nonnegative
integer $m$ such that $G^m=G^{m+1}=\cdots$, which then implies
$\bar{G}=G^m$. In particular, for a graph with $n$ vertices, since it
has at most $n(n-1)/2$ edges, its triangular closure can be obtained
in at most $n(n-1)/2$ steps.

\begin{remark}
  For readers familiar with graph theory, \cref{def:tri.closure} is mathematically equivalent to the standard definition of \emph{transitive closure}, and the equivalence will become transparent in the proof of \cref{thm:connectivity.controllability}. The triangular closure we introduce here imitates the computations of graded Lie brackets/algebras in a more natural way, so that all orders of Lie brackets can be calculated in a graph.
\end{remark}

Recall from \cref{lem:lie-graph} that given a subset $\Gamma\subseteq\mathcal{B}$ and its associated graph $G=\tau(\Gamma)$, taking first-order Lie brackets of the elements in $\Gamma$ corresponds to adding edges that connect the endpoints of incident edges in $G$. Applying this procedure to $G=G^{0}$, as defined in \cref{def:tri.closure}, exactly results in $G^1$. Inductively, successively Lie bracketing the elements in $\Gamma$ up to order $m$ will generate the graph $G^m$, as shown
below.

\begin{theorem}
  \label{thm:lie-graph}
  Given a subset $\Gamma\subseteq\mathcal{B}$, let $\Gamma^0\subseteq\Gamma^1\subseteq\cdots$ be an ascending chain of subsets of $\mathcal{B}$ such that $\Gamma^0=\Gamma$, $\Gamma^1=[\Gamma^{0},\Gamma^{0}]\bigcup\Gamma^{0}$, $\dots$, $\Gamma^{m+1}=[\Gamma^{m},\Gamma^{m}]\bigcup\Gamma^{m}$, $\dots$, where $[\Gamma^{m},\Gamma^{m}]=\{[A,B]:A,B\in\Gamma^{m}\}$. Then $G^m=\tau(\Gamma^m)$ holds for all $m=0,1,\ldots$
\end{theorem}

\begin{proof}
  This follows immediately from the definitions of $G^m$ and $\Gamma^m$.
\end{proof}

Recall that for any finite $G\in\mathcal{G}$, $G^m$ stabilizes to
$\bar{G}$ in finite steps. Meanwhile, by \cref{thm:lie-graph},
$\Gamma^m$ also stabilizes to a subset
$\hat{\Gamma}\subseteq\mathcal{B}$ which must satisfy
$\bar G=\tau(\hat{\Gamma})$. Intuitively, $\hat{\Gamma}$ is supposed
to contain all the elements that can be generated by the iterated Lie
brackets of the elements in $\Gamma$, because $\bar{G}$ is the largest
graph generated by $G$. This conclusion is then rigorously verified in
the following corollary.

\begin{corollary}
  \label{cor:lie-graph}
  Let $\Gamma$ be a subset of $\mathcal{B}$ and $G=\tau(\Gamma)$ be
  the graph associated with $\Gamma$. If
  $\hat\Gamma\subseteq\mathcal{B}$ satisfies
  $\tau(\hat\Gamma)=\bar{G}$, then
  $\lie(\Gamma)=\mathrm{span}\,(\hat\Gamma)$.
\end{corollary}

\begin{proof}
  Let $m$ be a nonnegative integer satisfying $G^m=\bar{G}$, then
  \cref{thm:lie-graph} implies that $\hat\Gamma=\Gamma^m$, hence
  $\Gamma^r=\hat{\Gamma}$ holds for all $r\geqslant{}m$.
  Consequently, by the definition of $\lie(\Gamma)$, we have
  $\lie(\Gamma) =\mathrm{span}\,(\bigcup_{i=0}^\infty\Gamma^i)
  =\mathrm{span}\,(\Gamma^m) =\mathrm{span}\,(\hat\Gamma)$.
\end{proof}

For the purpose of controllability analysis, the subsets of $\mathcal{B}$ generating the whole Lie algebra $\mathfrak{so}(n)$ is of great interest. Therefore, we characterize such subsets by their associated graphs below, 
which is also a special
case of \cref{cor:lie-graph}.

\begin{corollary}
 \label{cor:completion}
 Consider a subset $\Gamma\subseteq\mathcal{B}$ with the associated graph $G=\tau(\Gamma)$, then $\lie(\Gamma)=\mathfrak{so}(n)$ if and only if $\bar{G}=K_n$.
\end{corollary}

\begin{proof}
  (Sufficiency): Let $\hat{\Gamma}\subseteq\mathcal{B}$ satisfy
  $\bar{G}=\tau(\hat{\Gamma})=K_n$, then the properties (i) and (iii)
  in \cref{prop:property_tau} imply
  $\hat\Gamma=\mathcal{B}$. Consequently,
  $\lie(\Gamma)=\mathrm{span}\,
  (\hat\Gamma)=\mathrm{span}\,(\mathcal{B})=\mathfrak{so}(n)$ by
  \cref{cor:lie-graph}.
    
  (Necessity): If $\lie(\Gamma)=\mathfrak{so}(n)$, then there exists
  some nonnegative integer $m$ such that $\Gamma^m=\mathcal{B}$. By
  \cref{thm:lie-graph}, we obtain
  $\bar{G}\supseteq{}G^m=\tau(\Gamma^m)=K_n$. On the other hand,
  because of $\bar{G}\subseteq K_n$, we conclude $\bar{G}=K_n$.
\end{proof}

Furthermore, \cref{cor:completion} sheds light on a graph
representation of controllability, which in turn can be characterized
in terms of graph connectivity. In the following section, we will
rigorously investigate this observation.

\subsubsection{Controllability Characterization in Terms of Graph Connectivity} 
\label{sec:graph_controllability}

The relationship between Lie brackets and graph operations developed in Section~\ref{sec:Lie_graph} enables us to employ graph theory techniques to analyze controllability of systems on $\son$ as in \cref{eq:system_SOn}. In particular, motivated by the connection between a Lie subalgebra and its associated graph presented in \cref{cor:completion}, controllability can be analyzed through the notion of triangular closure defined in \cref{def:tri.closure}.

\begin{proposition}
  \label{prop:controllability}
  The bilinear system in \cref{eq:system_SOn} is controllable on
  $\son$ if and only if $\overline{\tau(\Gamma)}=K_n$, where $\tau$ is
  defined as in \cref{eq:graph-map},
  $\Gamma=\{\Omega_{i_0j_0},\dots,\Omega_{i_mj_m}\}$, and $K_n$ is a
  complete graph of $n$ vertices.
\end{proposition}

\begin{proof}
  By the LARC shown in \cref{thm:LARC}, the system in
  \cref{eq:system_SOn} is controllable on $\mathrm{SO}(n)$ if and only
  if $\lie(\Gamma)=\mathfrak{so}(n)$, which is equivalent to
  $\overline{\tau(\Gamma)}=K_n$ by \cref{cor:completion}.
\end{proof}

Using the following two examples, we will verify \cref{prop:controllability} and draw a parallel between examining the LARC and generating triangular closure of the graph associated with the considered system. This comparison in turn illuminates a graphic visualization of the algebraic procedure of generating Lie algebras for the set of drift and control vector fields.

\begin{example}
  \label{ex:controllable}
  Consider the system on $\mathrm{SO}(4)$ given by 
  \begin{equation}
    \label{eq:controllable.so4}
    \dot{X}(t)=(u_1\Omega_{12}+u_2\Omega_{23}+u_3\Omega_{13}+u_4\Omega_{34})X(t),\quad X(0)=I.
  \end{equation}
  Applying $\tau$ to the set of the control vector fields $\Gamma$ results in its associated graph $G=(V,E)$ as follows,
  \[
    \Gamma=\{\Omega_{12},\Omega_{23},\Omega_{13},\Omega_{34}\}
    \xleftrightarrow{\tau} \{v_1v_2,v_2v_3,v_1v_3,v_3v_4\}=E.
  \]
  Because the first order Lie brackets
  $[\Omega_{23},\Omega_{34}]=\Omega_{24}$ and
  $[\Omega_{13},\Omega_{34}]=\Omega_{14}$ are not in $\Gamma$, we have
  $\Gamma^1=\Gamma\cup\{\Omega_{24},\Omega_{14}\}$. Correspondingly,
  according to \cref{cor:lie-graph}, $G^1=(V,E^{1})$ can be obtained
  by applying $\tau$ to $\Gamma^1$, i.e., 
  \[
    \Gamma^1=\Gamma\cup\{\Omega_{24},\Omega_{14}\}
    \xleftrightarrow{\tau} \{v_2v_4,v_1v_4\}\cup E=E^{1}.
  \]
  Notice that $\mathrm{span}\,(\Gamma^1)=\mathfrak{so}(4)$ and
  simultaneously $G^1=\bar{G}=K_4$, which concludes controllability of
  the system in \cref{eq:controllable.so4} from both algebraic and
  graph-theoretic perspectives. The graphs $G$ and $G^1$ are shown in
  \cref{fig:controllable.so4}. In particular, the two red edges in
  $G^1$, which are not in $G$, correspond to the elements in
  $[\Gamma,\Gamma]$.
  
  \begin{figure}[tbh]
    \centering
    \begin{tikzpicture}[baseline=5ex,thick]
      \node[xshift=4*\R,left] at (0,1) {$v_1$};
      \node[xshift=4*\R,right] at (1,1) {$v_2$};
      \node[xshift=4*\R,right] at (1,0) {$v_3$};
      \node[xshift=4*\R,left] at (0,0) {$v_4$};
      
      \node[left] at (0,1) {$v_1$};
      \node[right] at (1,1) {$v_2$};
      \node[right] at (1,0) {$v_3$};
      \node[left] at (0,0) {$v_4$};
      
      \node[below] at (0.5,-0.2) {$G$};
      \node[xshift=4*\R,below] at (0.5,-0.2) {$G^{1}$};
      
      \draw (0,1) -- (1,1);
      \draw (1,1) node[vertex]{} -- (1,0);
      \draw (1,0) -- (0,1) node[vertex]{};
      \draw (1,0) node[vertex]{} -- (0,0) node[vertex]{};

      \node[xshift=2*\R] at (0.5,0.5) {$\dashrightarrow$};
      
      \draw[xshift=4*\R, red] (0,0) -- (0,1);
      \draw[xshift=4*\R, red] (0,0) -- (1,1);
      \draw[xshift=4*\R] (0,1) -- (1,1);
      \draw[xshift=4*\R] (1,1) node[vertex]{} -- (1,0);
      \draw[xshift=4*\R] (1,0) -- (0,1) node[vertex]{};
      \draw[xshift=4*\R] (1,0) node[vertex]{} -- (0,0) node[vertex]{};
    \end{tikzpicture}
    \caption{The graph $G$ associated with the 
        system~\cref{eq:controllable.so4} in
        \cref{ex:controllable} and its triangular closure
        $G^1$.  Note that the red edges in $G^1$ correspond to the
        vector fields generated by the first-order Lie brackets of the
        control vector fields in $\Gamma$.}
    \label{fig:controllable.so4}
  \end{figure}
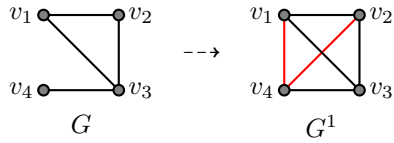
\end{example}

\Cref{ex:controllable} presents a controllable system whose associated graph has a complete triangular closure, which in turn validates the sufficiency of \cref{prop:controllability}. The necessity is illustrated using the following example through an uncontrollable system.

\begin{example}
  \label{ex:uncontrollable.so5}
  Consider the system on $\mathrm{SO}(5)$ driven by three control
  inputs, given by
  \begin{equation}
    \label{eq:uncontrollable.so5}
    \dot{X}(t)=(u_1\Omega_{12}+u_2\Omega_{23}+u_3\Omega_{34})X(t),\quad X(0)=I,
  \end{equation}
  and let $\Gamma=\{\Omega_{12},\Omega_{23},$ $\Omega_{34}\}$ denote
  the set of control vector fields. Some straightforward calculations
  yield the Lie algebra
  $\lie(\Gamma)={\mathrm{span}}\,\{\Omega_{12},\Omega_{23},
  \Omega_{34},\Omega_{13},\Omega_{14},\Omega_{24}\}$, which has
  dimension $6$. Therefore, the system in \cref{eq:uncontrollable.so5}
  is not controllable, since $\dim\mathfrak{so}(5)=10$. Using the
  graph-theoretic approach, \cref{fig:uncontrollable.so5} shows the
  procedure of generating $\bar{H}$ from $H=\tau(\Gamma)$. In
  particular, $\bar{H}=H^2$ shown in \cref{fig:uncontrollable.so5} is
  not complete, which verifies the necessity of
  \cref{prop:controllability}.

  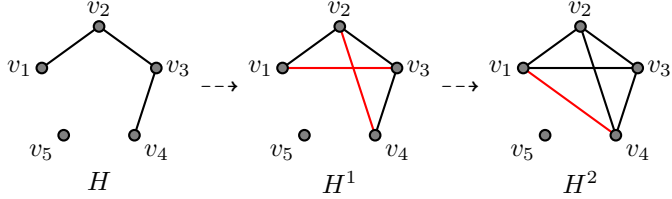
\begin{figure}[tbh]
    \centering
    \begin{tikzpicture}[baseline=-5ex,thick]
      \node[below] at (270:1cm) {$H$};
      \node[xshift=4*\R, below] at (270:1cm) {$H^1$};
      \node[xshift=8*\R, below] at (270:1cm) {$H^2$};
      
      \node[left] at (162:\R) {$v_1$};
      \node[above] at (90:\R) {$v_2$};
      \node[right] at (18:\R) {$v_3$};
      \node[below right] at (306:\R) {$v_4$};
      \node[below left] at (234:\R) {$v_5$};
      
      \node[xshift=4*\R, left] at (162:\R) {$v_1$};
      \node[xshift=4*\R, above] at (90:\R) {$v_2$};
      \node[xshift=4*\R, right] at (18:\R) {$v_3$};
      \node[xshift=4*\R, below right] at (306:\R) {$v_4$};
      \node[xshift=4*\R, below left] at (234:\R) {$v_5$};
      
      \node[xshift=8*\R, left] at (162:\R) {$v_1$};
      \node[xshift=8*\R, above] at (90:\R) {$v_2$};
      \node[xshift=8*\R, right] at (18:\R) {$v_3$};
      \node[xshift=8*\R, below right] at (306:\R) {$v_4$};
      \node[xshift=8*\R, below left] at (234:\R) {$v_5$};      
      
      \draw (162:\R) node[vertex]{} -- (90:\R);
      \draw (90:\R) node[vertex]{} -- (18:\R);
      \draw (18:\R) node[vertex]{} -- (306:\R) node[vertex]{};
      \draw (234:\R) node[vertex]{};

      \node[xshift=2*\R] at (0,0) {$\dashrightarrow$};

      \draw[xshift=4*\R, red] (90:\R) -- (306:\R);
      \draw[xshift=4*\R, red] (162:\R) -- (18:\R);
      \draw[xshift=4*\R] (162:\R) node[vertex]{} -- (90:\R);
      \draw[xshift=4*\R] (90:\R) node[vertex]{} -- (18:\R);
      \draw[xshift=4*\R] (18:\R) node[vertex]{} -- (306:\R) node[vertex]{};
      \draw[xshift=4*\R] (234:\R) node[vertex]{};

      \node[xshift=6*\R] at (0,0) {$\dashrightarrow$};

      \draw[xshift=8*\R, red] (162:\R) -- (306:\R);
      \draw[xshift=8*\R] (90:\R) -- (306:\R);
      \draw[xshift=8*\R] (162:\R) -- (18:\R);
      \draw[xshift=8*\R] (162:\R) node[vertex]{} -- (90:\R);
      \draw[xshift=8*\R] (90:\R) node[vertex]{} -- (18:\R);
      \draw[xshift=8*\R] (18:\R) node[vertex]{} -- (306:\R) node[vertex]{};
      \draw[xshift=8*\R] (234:\R) node[vertex]{};
    \end{tikzpicture}
    \caption{The graph visualization of Lie bracketing control vector
      fields of the system in \cref{eq:uncontrollable.so5} in
      \cref{ex:uncontrollable.so5}. Specifically, the graph $H$ is
      associated with the set of control vector fields, $H^1$
      visualizes the first-order Lie brackets, and $H^2$ is the
      triangular closure of $H$. Note that the red edges correspond to
      the vector fields in $\Gamma$ generated by Lie brackets.}
    \label{fig:uncontrollable.so5}
  \end{figure}
\end{example}

It is worth noting that the graph $G$ in \cref{fig:controllable.so4}
associated with the controllable system in \cref{eq:controllable.so4}
is connected, but the graph $H$ in \cref{fig:uncontrollable.so5}
associated with the uncontrollable system in
\cref{eq:uncontrollable.so5} is not. This observation inspires the
characterization of controllability for systems on $\son$ by graph
connectivity.

\begin{theorem}
  \label{thm:connectivity.controllability}
  The system in \cref{eq:system_SOn} is controllable on $\son$ if and
  only if $\tau(\Gamma)$ is connected, where
  $\Gamma=\{\Omega_{i_0 j_0},\dots,\Omega_{i_m j_m}\}$ and
  $\tau(\Gamma)$ is the graph associated with $\Gamma$.
\end{theorem}

\begin{proof}
  Owing to \cref{prop:controllability}, it suffices to prove that the
  triangular closure of $\tau(\Gamma)$ is complete if and only if
  $\tau(\Gamma)$ is connected.
  
  (Sufficiency): Suppose that $G=\tau(\Gamma)=(V,E)$ is connected,
  then there is a path in $G$ from $v_i$ to $v_j$ for any
  $v_i,v_j\in{}V$, say $v_iw_1w_2\cdots{}w_kv_j$ with
  $w_1,\ldots,w_k\in{}V$. Therefore, we have
  $v_iw_2\in{}E^1,\ldots, v_iw_k\in{}E^{k-1}$ and
  $v_iv_j\in{}E^{k}\subseteq\bar{E}$. Since $v_i,v_j\in V$ are chosen
  arbitrarily, we conclude that the triangular closure $\bar{G}$
  contains all edges $v_iv_j$, hence $\bar{G}=K_n$. In addition, this
  process of generating $\bar G$ is illustrated in
  \cref{fig:proof.completeness} with the case of $k=5$.

  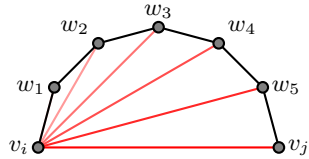
\begin{figure}[htbp]
    \centering
    \small
    \begin{tikzpicture}[thick]
      \node[left] at (180:2*\R) {$v_i$};
      \node[left] at (150:2*\R) {$w_1$};
      \node[above left] at (120:2*\R) {$w_2$};
      \node[above] at (90:2*\R) {$w_3$};
      \node[above right] at (60:2*\R) {$w_4$};
      \node[right] at (30:2*\R) {$w_5$};
      \node[right] at (0:2*\R) {$v_j$};
      
      \draw[red!25] (180:2*\R) -- (150:2*\R);
      \draw[red!40] (180:2*\R) -- (120:2*\R);
      \draw[red!55] (180:2*\R) -- (90:2*\R);
      \draw[red!70] (180:2*\R) -- (60:2*\R);
      \draw[red!85] (180:2*\R) -- (30:2*\R);
      \draw[red!] (180:2*\R) -- (0:2*\R);
    
      \draw \foreach \x in {0,1, ...,5}
      {
        (30*\x:2*\R) node[vertex]{} -- (30*\x+30:2*\R)
      } (180:2*\R) node[vertex]{};
    \end{tikzpicture}
    \caption{Illustration of the proof of sufficiency of
      \cref{thm:connectivity.controllability}.}
    \label{fig:proof.completeness}
  \end{figure}

  (Necessity): We assume that the triangular closure $\bar G$ of
  $G=\tau(\Gamma)$ is complete. If there exists an edge $v_iv_j$
  \emph{not} in $G$, since $v_iv_j$ is in $\bar{G}=(V,\bar{E})$, we
  may then assume $v_iv_j\in{}E^k$ and $v_iv_j\not\in{}E^{k-1}$ for
  some positive integer $k$. Hence, by \cref{def:tri.closure}, there
  is some vertex $w_1$ such that $v_iw_1,w_1v_j\in{}E^{k-1}$, i.e.,
  there exists a path $v_iw_1v_j$ in $G^{k-1}$ connecting $v_i$ and
  $v_j$. Repeating this procedure results in a path in $G$ connecting
  $v_i$ and $v_j$, which implies the connectivity of $G$, and hence
  the proof is done.
\end{proof}

\begin{remark}
  In addition to controllability characterization, \cref{thm:connectivity.controllability} highlights a crucial property of the map $\tau$, that is, $\lie(\Gamma)=\mathfrak{so}(n)$
  for some $\Gamma\subseteq\mathcal{B}$ \emph{if and only if} $\tau(\Gamma)$ is \emph{connected}, which is an equivalent formulation of \cref{thm:connectivity.controllability}.
\end{remark}

Because a connected graph with $n$ vertices contains at least $n-1$ edges, \cref{thm:connectivity.controllability} also identifies the minimum number of control inputs for the system in \cref{eq:system_SOn} to be controllable, as identified using the symmetric group method presented in \cref{thm:SOn_Sn,rem:number_of_control}.

\begin{corollary}
  \label{cor:control_number}
  If a system on $\son$ in \cref{eq:system_SOn} is controllable, then
  the number of control inputs $m$ is at least $n-2$, i.e.,
  $m\geqslant{}n-2$.
\end{corollary}

Although \cref{thm:connectivity.controllability} is developed to
examine controllability, it also helps establish some general facts in
graph theory from the control systems perspective. In the following,
we present one such result that is related to triangular
closures. This property also plays an important role in characterizing
controllable submanifolds for uncontrollable systems by connected
component of the graph associated with the control system.

\begin{lemma}
  \label{lem:tri.closure.complete.components}
  The triangular closure $\bar{G}$ of a graph $G$ is a disjoint union
  of its complete components.
\end{lemma}

\begin{proof}
  The proof is a direct application of the proof of
  \cref{thm:connectivity.controllability} to each connected component
  of $G$.
\end{proof}

By the above \cref{lem:tri.closure.complete.components}, we can adopt our main result in \cref{thm:connectivity.controllability} to study an uncontrollable system by taking the triangular closure of its associated graph, which is the union of the triangular closures of all connected components.

\begin{theorem}
  \label{thm:controllable.submanifold}
  The controllable submanifold of the system in \cref{eq:system_SOn}
  is determined by the connected components of its associated graph.
\end{theorem}

\begin{proof}
  Let $\Gamma\subseteq\mathcal{B}$ be the set of vector fields
  governing the dynamics of the system in \cref{eq:system_SOn},
  $G=\tau(\Gamma)$ be the graph representation of $\Gamma$, and
  $\bar{G}$ denote the triangular closure of $G$. Since connected
  components of $G$ determine the complete components of $\bar{G}$, it
  suffices to show that the controllable submanifold of the system is
  determined by the complete components of $\bar{G}$.

  According to the Frobenius Theorem~\cite{warner.lie.groups}, the
  controllable submanifold of the system in \cref{eq:system_SOn} is
  the maximal integral submanifold of $\lie(\Gamma)$ passing through
  the identity matrix $I$. Hence, by
  \cref{lem:tri.closure.complete.components}, because of the
  completeness of each component of $\bar{G}$, the set
  $\tau^{-1}(\bar{G})\subseteq\mathcal{B}$ is closed under Lie
  bracket, which implies
  $\mathrm{span}\,\tau^{-1}(\bar{G})=\lie(\Gamma).$ Therefore, we
  conclude that $\lie(\Gamma)$, and thus its maximal integral
  submanifold, is determined by $\bar G$.
\end{proof}

\Cref{thm:controllable.submanifold} further reveals a one-to-one correspondence between the Lie algebra generated by a subset of $\mathcal{B}$ and the triangular closure of its associated graph in $\mathcal{G}$. Leveraging this one-to-one correspondence, we are able to give an explicit characterization of controllable submanifolds for uncontrollable systems in terms of connected components of their associated graphs.

\begin{example}[Controllable Submanifold] \label{ex:submanifold}
  Consider two bilinear systems defined on $\mathrm{SO}(6)$ in the form of \cref{eq:system_SOn} governed by the vector fields $\Gamma_1=\{\Omega_{12},\Omega_{23},\Omega_{45},\Omega_{46}\}$ and
  $\Gamma_2=\{\Omega_{13},\Omega_{23},\Omega_{46},\Omega_{56}\}$, respectively. \Cref{fig:controllable.submanifolds} shows their associated graphs $G_1=\tau(\Gamma_1)$ and $G_2=\tau(\Gamma_2)$, neither of which is connected. Therefore, by \cref{thm:lie-graph}, both systems are not controllable on $\mathrm{SO}(6)$. On the other hand, we notice that $\overline{G_1}=\overline{G_2}$. So by \cref{thm:controllable.submanifold}, the two systems have the same   controllable submanifold. Specifically, the controllable submanifold is the Lie subgroup of $\mathrm{SO}(6)$ with the Lie algebra
  \[
    \lie(\Gamma_1)=\lie(\Gamma_2)
    =\mathrm{span}\,\{\Omega_{ij}:1\leqslant{}i<j\leqslant{}3\}
    \oplus\mathrm{span}\,\{\Omega_{ij}:4\leqslant{}i<j\leqslant{}6\}.
  \]
  Moreover, both $\overline{G_1}$ and $\overline{G_2}$ contain two complete components with the vertex sets $U=\{v_1,v_2,v_3\}$ and $W=\{v_4,v_5,v_6\}$, which are also the vertex sets of the connected components of $G_1$ (or $G_2$). It then follows that the Lie algebra of the controllable submanifold, $\mathrm{span}\,\{\Omega_{ij}:v_i,v_j\in{}U\}
  \oplus\mathrm{span}\,\{\Omega_{ij}:v_i,v_j\in{}W\}$, can be explicitly characterized by the vertex sets of the complete components of $\overline{G_1}$ and $\overline{G_2}$, as well as the
  connected components of $G_1$ and $G_2$.
  
  \begin{figure}[tbh]
    \centering
    \tikzstyle{r}=[xshift=5*\R]
    \begin{tikzpicture}[thick]
      \node[left] at (180:\R) {$v_1$};
      \node[above left] at (120:\R) {$v_2$};
      \node[above right] at (60:\R) {$v_3$};
      \node[right] at (0:\R) {$v_4$};
      \node[below right] at (300:\R) {$v_5$};
      \node[below left] at (240:\R) {$v_6$};

      \node[r, left] at (180:\R) {$v_1$};
      \node[r, above left] at (120:\R) {$v_2$};
      \node[r, above right] at (60:\R) {$v_3$};
      \node[r, right] at (0:\R) {$v_4$};
      \node[r, below right] at (300:\R) {$v_5$};
      \node[r, below left] at (240:\R) {$v_6$};

      \node[below] at (270:1cm) {$G_1$};
      \node[r, below] at (270:1cm) {$\overline{G_1}$};      

      \draw (120:\R) -- (60:\R) node[vertex]{};
      \draw (180:\R) node[vertex]{} -- (120:\R) node[vertex]{};
      \draw (240:\R) node[vertex]{} -- (0:\R);
      \draw (0:\R) node[vertex]{} -- (300:\R) node[vertex]{};

      \node[xshift=2.5*\R] at (0,0) {$\dashrightarrow$};
      
      \draw[r, red] (60:\R) -- (180:\R);
      \draw[r, red] (240:\R) -- (300:\R);
      \draw[r] (120:\R) -- (60:\R);
      \draw[r] (120:\R) -- (60:\R) node[vertex]{};
      \draw[r] (180:\R) node[vertex]{} -- (120:\R) node[vertex]{};
      \draw[r] (240:\R) node[vertex]{} -- (0:\R);
      \draw[r] (0:\R) node[vertex]{} -- (300:\R) node[vertex]{};
    \end{tikzpicture}

    \medskip
    \begin{tikzpicture}[baseline=-5ex,thick]
      \node[left] at (180:\R) {$v_1$};
      \node[above left] at (120:\R) {$v_2$};
      \node[above right] at (60:\R) {$v_3$};
      \node[right] at (0:\R) {$v_4$};
      \node[below right] at (300:\R) {$v_5$};
      \node[below left] at (240:\R) {$v_6$};

      \node[r, left] at (180:\R) {$v_1$};
      \node[r, above left] at (120:\R) {$v_2$};
      \node[r, above right] at (60:\R) {$v_3$};
      \node[r, right] at (0:\R) {$v_4$};
      \node[r, below right] at (300:\R) {$v_5$};
      \node[r, below left] at (240:\R) {$v_6$};
      
      \node[below] at (270:1cm) {$G_2$};
      \node[r, below] at (270:1cm) {$\overline{G_2}$};

      \draw (180:\R) node[vertex]{} -- (60:\R);
      \draw (60:\R) node[vertex]{} -- (120:\R) node[vertex]{};
      \draw (0:\R) node[vertex]{} -- (240:\R);
      \draw (240:\R) node[vertex]{} -- (300:\R) node[vertex]{};

      \node[xshift=2.5*\R] at (0,0) {$\dashrightarrow$};

      \draw[r, red] (180:\R) -- (120:\R);
      \draw[r, red] (300:\R) -- (0:\R);
      \draw[r] (180:\R) node[vertex]{} -- (60:\R);
      \draw[r] (60:\R) node[vertex]{} -- (120:\R) node[vertex]{};
      \draw[r] (0:\R) node[vertex]{} -- (240:\R);
      \draw[r] (240:\R) node[vertex]{} -- (300:\R) node[vertex]{};
    \end{tikzpicture}
    \caption{The graphs and their triangular closures associated with
      the systems in \cref{ex:submanifold}. Specifically, the graphs
      $G_1$ and $G_2$ on the left are associated with the systems
      governed by $\Gamma_1$ and $\Gamma_2$, respectively, and their
      triangular closures $\overline{G_1}$ and $\overline{G_2}$ are on
      the right. Red edges correspond to vector fields generated by
      Lie brackets.}
    \label{fig:controllable.submanifolds}
  \end{figure}
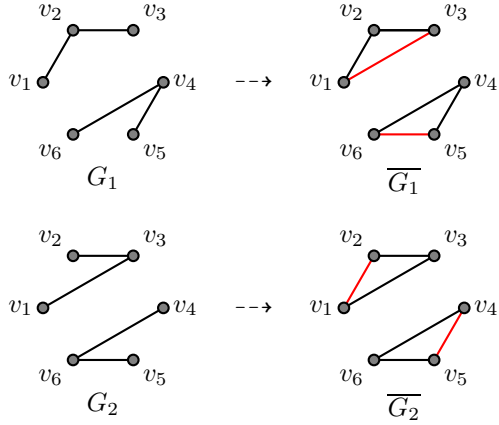
\end{example}

Furthermore, the developed method of characterizing controllability in terms of graph connectivity is not constrained to systems defined on Lie groups. In particular, as shown in the following example, it can be applied to study formation control of multi-agent systems defined on graphs. From an algebraic perspective, it is equivalent to using the graph-theoretic method to analyze the Lie algebra generated by \emph{symmetric matrices}.

\begin{example}[Formation Control]
  \label{ex:formation.control}
  In this example, we consider formation control of a multi-agent system, which concerns with the question of coordinating the system to a consensus state. For such a purpose, the dynamics of each agent in a network of $N$ agents with the coupling topology given by the graph $G=(V,E), V=\{v_1,\ldots,v_N\}$, is generally represented by 
  \begin{equation}
    \label{eq:dynamic.law}
    \dot{x}_i(t)=\sum_{j\in{}V(i)} u_{ij}(x_j-x_i), \quad{}
    1\leqslant{} i\leqslant{} N,
  \end{equation}
  where $x_i(t)\in\mathbb{R}^n$ denotes the state of the $i$-th agent, $V(i)=\{1\leqslant{} j\leqslant{} N: v_iv_j\in E\}$ denotes the set of neighboring agents of $i$, and $u_{ij}=u_{ji}$ are the external inputs that control the reciprocal interaction between the $i$-th and $j$-th agents~\cite{Chen14}.
  
  We will first formulate the dynamic law in \cref{eq:dynamic.law} into a matrix form, and then apply our analysis on a Lie algebra associated with it. To do this, let
  \[
    A_{ij}:=E_{ii}+E_{jj}-E_{ij}-E_{ji}
  \]
  be an $N$-by-$N$ symmetric matrix with zero row and column sums, and let $X\in\mathbb{R}^{N\times{}n}$ denote a matrix whose row vectors are the states of the agents: 
  \[
    X=\begin{bmatrix}
      x_1^{\tr} \\
      \vdots \\
      x_N^{\tr}
    \end{bmatrix}.
  \]
    Then, we can rewrite \cref{eq:dynamic.law} into the following matrix form,
  \begin{equation}
    \label{eq:dynamic.law.matrix}
    \dot{X}=\sum_{v_iv_j\in{}E}u_{ij}A_{ij}X.
  \end{equation}
  
  The formation controllability of the multi-agent system in \cref{eq:dynamic.law.matrix} is determined by the LARC~\cite{Chen14}. Thus, to study this system, we need to know the algebraic structure of matrices $\{A_{ij}\}$. Observe that for $B_{ijk}=-(\Omega_{ij}+\Omega_{jk}+\Omega_{ki})$ and distinct indices $1\leqslant{} i,j,k,l,m\leqslant{} N$, we have
  \begin{equation}
    \label{eq:formation.control.algebraic.property}
    \begin{aligned}
      [A_{ij}, A_{jk}]   &= B_{ijk}, \\
      [B_{ijk}, A_{ij}]  &= 2(A_{ik}-A_{jk}), \\
      [B_{ijk}, A_{il}]  &= -A_{ij}+A_{jl}+A_{ik}-A_{kl}, \\
      [B_{ijk}, B_{ijl}] &= B_{ikl}+B_{jkl}, \\
      [B_{ijk}, B_{ilm}] &= B_{jlm}+B_{kml} = B_{lkj}+B_{mjk}. 
    \end{aligned}
  \end{equation}
  Therefore, the Lie algebra $\mathfrak{g}:=\lie\{A_{ij}\}$ has a decomposition,
  $\mathfrak{g}=\mathfrak{g}_1\oplus{}\mathfrak{g}_{-1}$, with
  $\mathfrak{g}_1=\mathrm{span}\,\{B_{ijk}\}$ and
  $\mathfrak{g}_{-1}=\mathrm{span}\,\{A_{ij}\}$. As a consequence, by the LARC, controllability of system~\cref{eq:dynamic.law.matrix} depends on whether the set $\Gamma:=\{A_{ij}: (i,j)\in{}E\}$ generates the Lie algebra $\mathfrak{g}$. Similar to bilinear systems on $\son$, we can adopt a graph-theoretic method for
  $\mathfrak{g}$ by associating \emph{one part of} $\mathfrak{g}$, i.e., $\mathfrak{g}_{-1}$, to a graph, which in the case of this example, \emph{coincides} with the graph on which the system is defined. To
  be more specific, for a complete graph $K_N$ and its set of edges $E$, we may define a map $\tau:\Gamma\to{}E$, which sends $A_{ij}\in\Gamma$ to $v_iv_j\in{}E$, so that the image of $\Gamma$ is exactly the graph $G$. Following the correspondence $\tau$, for two adjacent edges $v_iv_j$ and $v_jv_k$, since $\lie\{A_{ij},A_{jk}\} =\lie\{A_{ij},A_{jk},A_{ki}\} =\mathrm{span}\,\{A_{ij}, A_{jk}, A_{ki}, B_{ijk}\}$, the triangle
  with edges $v_iv_j, v_jv_k, v_kv_i$ in $G$ represents the Lie subalgebra spanned by $\{A_{ij}, A_{jk}, A_{ki}, B_{ijk}\}$. More generally, by the algebraic relations in \cref{eq:formation.control.algebraic.property}, any triangularly closed subgraph of $K_N$ is associated with a subalgebra of $\mathfrak{g}$. Therefore, the Lie (sub)algebra generated by $\Gamma$ can be represented by the triangular closure of $G$; and if $G$ is connected, then its triangular closure is complete, which
  suggests that $\lie(\Gamma)$ contains all $A_{ij}$'s, so we have $\lie(\Gamma)=\mathfrak{g}$. In conclusion, the controllability of system~\cref{eq:dynamic.law}, and equivalently, system~\cref{eq:dynamic.law.matrix}, is determined by graph connectivity of $G$.
\end{example}

By now, we have conducted a detailed investigation into controllability of bilinear systems on $\son$ governed by the standard basis elements of $\mathfrak{so}(n)$. Before we extend the scope of our investigation to general bilinear systems, we show that, in contrast to \cref{cor:control_number}, a driftless bilinear system on $\son$ can be controllable using only two control inputs, for all $n>0$.

\begin{example}
  Recall that by \cref{cor:control_number}, driftless bilinear systems
  on $\mathrm{SO}(n)$ with control vector fields in the standard basis
  of $\mathfrak{so}(n)$ require at least $n-1$ inputs to be
  controllable. However, this conclusion may not hold for general
  systems governed by vector fields not in the standard basis. For
  example, the following system with two control inputs
  \begin{equation}
    \label{eq:minimal.control.inputs}
    \dot{X}(t)=\bigl[u_1(t)C_1+u_2(t)C_2\bigr]X(t),\quad X(0)=I,
  \end{equation}
  where $C_1=\Omega_{12}$ and $C_2=\sum_{i=1}^{n-1}\Omega_{i,i+1}$, is
  controllable on $\son$.  To see this, we will show
  $\Omega_{1k}\in\lie(\{C_1,C_2\})$ for any
  $2\leqslant{}k\leqslant{}n$ by induction. At first, note that
  $\Omega_{13}=[C_1,C_2]\in\lie(\{C_1,C_2\})$.  Next, we assume
  $\Omega_{12}, \Omega_{13},\ldots,\Omega_{1k}\in\lie(\{C_1,C_2\})$
  for some $3\leqslant{}k<n$, which is the induction
  hypothesis. Consequently, we have
  $[\Omega_{1k},C_2]=\Omega_{2k}-\Omega_{1,k-1}+\Omega_{1,k+1}$ and
  $[\Omega_{1k},C_1]=\Omega_{2k}$, which implies
  $\Omega_{1,k+1}=[\Omega_{1k},C_2]-[\Omega_{1k},C_1]+\Omega_{1,k-1}
  \in\lie(\{C_1,C_2\})$. By induction, we conclude
  $\Omega_{1k}\in\lie(\{C_1,C_2\})$ for any
  $2\leqslant{}k\leqslant{}n$. This result implies
  $\lie(\Sigma)\subseteq\lie(\{C_1,C_2\})$, where
  $\Sigma=\{\Omega_{1k}:1\leqslant k\leqslant n\}$. Obviously
  $\tau(\Sigma)$ is a connected graph, and hence by
  \cref{thm:connectivity.controllability},
  $\lie(\Sigma)=\mathfrak{so}(n)$, and the the system in
  \cref{eq:minimal.control.inputs} is thus controllable.
\end{example}

\subsection{Equivalence Between the Symmetric Group and Graph-Theoretic Methods} 
\label{sec:Sn_graph}

In Sections~\ref{sec:Sn} and~\ref{sec:graph}, we developed two combinatorics-based methods to analyze controllability of bilinear systems. Both methods connect the Lie brackets of vector fields to operations on combinatorial objects. We will show next that an equivalence exists between the symmetric group and the graph-theoretic method when systems on $\son$ are concerned. We first illustrate this equivalence through a controllable system on $\mathrm{SO}(4)$. 

\begin{example}
  \label{ex:revisit.controllable.so4}
  Let us revisit the system in \cref{eq:controllable.so4} in \cref{ex:controllable}, governed by the set of
  vector fields $\Gamma=\{\Omega_{12},\Omega_{23},\Omega_{13},\Omega_{34}\}$. We have shown therein that this system is controllable on ${\rm SO}(4)$ by using the graph-theoretic method; and for the symmetric group method, we may choose
  $\Sigma_1=\{\Omega_{12},\Omega_{13},\Omega_{34}\}\subset\Gamma$ so that $\iota(\Sigma_1)=(1342)$ is a $4$-cycle. However, $\Sigma_1$ is not the only subset that is related to a $4$-cycle, and, for example, one can easily verify that $\iota{}$ also relates the subsets
  $\Sigma_2=\{\Omega_{13}, \Omega_{23}, \Omega_{34}\}$ and
  $\Sigma_3=\{\Omega_{12}, \Omega_{23}, \Omega_{34}\}$ to $4$-cycles
  as $\iota(\Sigma_2)=(13)(23)(34)=(1342)$ and
  $\iota(\Sigma_3)=(12)(23)(34)=(1234)$. Moreover, it is worth noting
  that $\Sigma_1$, $\Sigma_2$, and $\Sigma_3$ are the only subsets of
  $\Gamma$ that are related to $4$-cycles. Meanwhile, and more
  importantly, their graph representations $\tau(\Sigma_1)$,
  $\tau(\Sigma_2)$, and $\tau(\Sigma_3)$ coincide with all three
  spanning trees of the graph $\tau(\Gamma)$ associated with the
  system (see \cref{tab:permutation-graph.controllable}). On the other
  hand, from the aspect of Lie algebra, we observe that $\Sigma_i$ is
  a \emph{minimal} subset of $\Gamma$ generating $\lie(\Gamma)$ for
  each $i=1,2,3$, that is, $\Sigma'=\Sigma_i$ for any
  $\Sigma'\subseteq\Sigma_i$ satisfying
  $\lie(\Sigma')=\lie(\Gamma)$. This observation sheds light on the
  general result: given a system on $\son$ governed by the set of
  vector fields $\Gamma$, if $\Sigma$ is a minimal subset of $\Gamma$
  with $\lie(\Sigma)=\lie(\Gamma)$, then $\iota(\Sigma)$ is an
  $n$-cycle if and only if $\tau(\Sigma)$ is a spanning tree of
  $\tau(\Gamma)$.
\end{example}

\begin{table}[tbp]
  \centering
  \begin{tabularx}{\textwidth}{lYl}
    \toprule
    Set of control vector fields
      & Graph
        & Permutation in $S_4$ \\
    \midrule
    $\Gamma=\{\Omega_{12},\Omega_{23},\Omega_{13},\Omega_{34}\}$
      & \begin{tikzpicture}[baseline={4mm},thick]
          \node[left] at (0,1) {$v_1$};
          \node[right] at (1,1) {$v_2$};
          \node[right] at (1,0) {$v_3$};
          \node[left] at (0,0) {$v_4$};
          
          \draw (0,1) -- (1,1);
          \draw (1,1) node[vertex]{} -- (1,0);
          \draw (1,0) -- (0,1) node[vertex]{};
          \draw (1,0) node[vertex]{} -- (0,0) node[vertex]{};
        \end{tikzpicture}
        & $\iota(\Gamma)=(12)(23)(13)(34)=(234)$ \\
    \midrule
    $\Sigma_1=\{\Omega_{12},\Omega_{13},\Omega_{34}\}$
      & \begin{tikzpicture}[baseline={4mm},thick]
          \node[left] at (0,1) {$v_1$};
          \node[right] at (1,1) {$v_2$};
          \node[right] at (1,0) {$v_3$};
          \node[left] at (0,0) {$v_4$};
          
          \draw (0,1) -- (1,1) node[vertex]{};
          \draw (1,0) -- (0,1) node[vertex]{};
          \draw (1,0) node[vertex]{} -- (0,0) node[vertex]{};
        \end{tikzpicture}
        & $\iota(\Sigma_1)=(12)(13)(34)=(1342)$ \\
    \midrule
    $\Sigma_2=\{\Omega_{13},\Omega_{23},\Omega_{34}\}$
      & \begin{tikzpicture}[baseline={4mm},thick]
          \node[left] at (0,1) {$v_1$};
          \node[right] at (1,1) {$v_2$};
          \node[right] at (1,0) {$v_3$};
          \node[left] at (0,0) {$v_4$};
          
          \draw (0,1) node[vertex]{} -- (1,0);
          \draw (1,1) node[vertex]{} -- (1,0);
          \draw (0,0) node[vertex]{} -- (1,0) node[vertex]{};
        \end{tikzpicture}
        & $\iota(\Sigma_2)=(13)(23)(34)=(1342)$ \\
    \midrule
    $\Sigma_3=\{\Omega_{12},\Omega_{23},\Omega_{34}\}$
      & \begin{tikzpicture}[baseline={4mm},thick]
          \node[left] at (0,1) {$v_1$};
          \node[right] at (1,1) {$v_2$};
          \node[right] at (1,0) {$v_3$};
          \node[left] at (0,0) {$v_4$};
          
          \draw (0,1) node[vertex]{} -- (1,1);
          \draw (1,1) node[vertex]{} -- (1,0);
          \draw (1,0) node[vertex]{} -- (0,0) node[vertex]{};
        \end{tikzpicture}
        & $\iota(\Sigma_3)=(12)(23)(34)=(1234)$ \\
    \bottomrule
  \end{tabularx}
  \caption{A comparison between two methods analyzing controllability:
    the symmetric groups method and the graph-theoretic method. Note
    that the graphs associated with $\Sigma_1$, $\Sigma_2$ and
    $\Sigma_3$ are \emph{spanning trees} of the associated graph of
    $\Gamma$, and that any tree is related to a $4$-cycle in the
    symmetric group $S_4$.}
  \label{tab:permutation-graph.controllable}
\end{table}

\begin{theorem}
  \label{thm:spanning.tree}
  Consider a bilinear system on $\son$ as in \cref{eq:system_SOn} and
  let $\Gamma\subseteq\mathcal{B}$ denote the set of vector fields
  governing the system dynamics.  Suppose $\Sigma\subseteq\Gamma$ is a
  \emph{minimal} subset such that $\iota(\Sigma)=\sigma\in S_n$ is an
  $n$-cycle (i.e., $\Sigma$ has no proper subset that is also related
  to an $n$-cycle via $\iota$), then its associated graph
  $\tau(\Sigma)$ is a \emph{spanning tree} of $\tau(\Gamma)$, and the
  system is therefore controllable.  Conversely, for a controllable
  system, any spanning tree $T$ of the connected graph $\tau(\Gamma)$
  corresponds to a subset $\Sigma'=\tau^{-1}(T)$, such that
  $\Sigma'\subseteq\Gamma$ is minimal and that $\iota(\Sigma')$ is an
  $n$-cycle in $S_n$.
  %
\end{theorem}

\begin{proof}
  From group theory we know that a \emph{minimal} $\Sigma$ with
  $\iota{}(\Sigma)$ being an $n$-cycle should consist of $n-1$
  transpositions, and that the union of the orbits of all $n-1$
  transpositions is the orbit of $\sigma$. This means the graph
  $\tau(\Sigma)$ has $n$ vertices and $n-1$ edges. Since a graph with
  $n$ vertices and $n-1$ edges is both \emph{connected} and
  \emph{acyclic}, and since $\tau(\Sigma)$ covers all $n$ vertices of
  $\tau(\Gamma)$, we conclude that $\tau(\Sigma)$ is the spanning tree
  of $\tau(\Gamma)$.

  On the other hand, for a subset $\Sigma'\subseteq\Gamma$ satisfying that $\tau(\Sigma')$ is a spanning tree of $\tau(\Gamma)$, we must have $|\Sigma'|=n-1$. Since a decomposition of an $n$-cycle needs at least $n-1$ transpositions, if $\iota(\Sigma')$ is an $n$-cycle, then $\Sigma'$ is obviously minimal. The following claim shows that $\iota(\Sigma')$ is indeed an $n$-cycle, regardless of the ordering of elements in $\Sigma'$.

  \begin{claim}
    A tree consisting of $k$ edges in the connected graph
    $\tau(\Gamma)$ in \cref{thm:spanning.tree} is related to a
    $(k+1)$-cycle via $\iota$, regardless of the ordering of
    transpositions.
  \end{claim}

  \begin{claimproof}
    Let us consider a tree $T$ with $k$ edges in $\tau(\Gamma)$, and
    prove the claim by induction. It is trivial for $k=1$; and for
    $k=2$, say $T=v_{j_1}v_{j_2}v_{j_3}$, then $\iota$ sends $T$ to
    either $(j_1j_2j_3)$ or $(j_1j_3j_2)$, depending on the orderings
    of $(j_1j_2)$ and $(j_2j_3)$. Assume the claim is true for $k=l-1$
    for some $l\in\mathbb{Z}_{+}$; and for a tree $T$ with $k=l$
    edges, we can choose a subtree $T'$ of $T$ that consists of $l-1$
    edges. Without loss of generality, we may assume $T'$ has vertices
    $\{v_1, \ldots, v_l\}$ and that $T$ has an additional vertex
    $v_{l+1}$ and an additional edge $v_1v_{l+1}$.  Let
    $\{\sigma_1, \ldots, \sigma_{l-1}\}$ be a set of transpositions
    such that each $\sigma_i$ is related to a distinct edge in $T'$ by
    $\iota$, and let $\rho=(1,l+1)$ denote the transposition related
    to the additional edge $v_1v_{l+1}$ in $T$. Our goal is to show
    that the permutation
    \begin{equation}
      \label{eq:permutation}
      \sigma_{i_1}\cdots\sigma_{i_t}\rho\sigma_{i_{t+1}}\cdots\sigma_{i_{l-1}}
    \end{equation}
    is an $(l+1)$-cycle. Note that for any
    $1\leqslant{} i,j_1,j_2\leqslant{} l$, we have the following law
    of commutation:
    \begin{equation}
      \label{eq:commutation.law}
      (i,l+1)(j_1j_2)=
      \begin{cases}
        (j_1j_2)(i,l+1) & \text{if neither $j_1$ or $j_2$ equals to $i$;} \\
        (j_1j_2)(j_2,l+1) & \text{if $j_1=i$.}
      \end{cases}
    \end{equation}
    Therefore, we can rewrite the permutation~\cref{eq:permutation} as
    $\sigma_{i_1}\cdots\sigma_{i_{l-1}}\rho'$, where $\rho'=(j_p,l+1)$
    for some $1\leqslant{} j_p\leqslant{} l$. By our assumption,
    $\sigma_{i_1}\cdots{}\sigma_{i_{l-1}}$ is an $l$-cycle:
    $\sigma_{i_1}\cdots{}\sigma_{i_{l-1}}=(j_1j_2\cdots{}j_l)$, so
    finally we have
    \[
      \begin{aligned}
        \sigma_{i_1}\cdots\sigma_{i_t}\rho\sigma_{i_{t+1}}\cdots\sigma_{i_{l-1}}
        &=\sigma_{i_1}\cdots{}\sigma_{i_{l-1}}\rho'
        =(j_1j_2\cdots{}j_l)(j_p,l+1) \\
        &=(j_1\cdots{} j_p,l+1,j_{p+1}\cdots{} j_l),
      \end{aligned}
    \]
    which is an $(l+1)$-cycle. It is clear that the ordering of
    transpositions is irrelevant in our proof.
  \end{claimproof}
  
  Therefore, a spanning tree
  $\tau(\Sigma')$ consisting of $n-1$ edges is related to an $n$-cycle
  in $S_n$ via $\iota$, which finishes our proof.
\end{proof}

Given a controllable system on $\son$, \cref{thm:spanning.tree} reveals the relation between $n$-cycles and spanning trees of the associated graph. In particular, for such a system governed by the set $\Gamma\subseteq\mathcal{B}$ of vector fields, this theorem supplements \cref{thm:connectivity.controllability} by explicitly describing the subsets of $\Gamma$ that are related to $n$-cycles using graphs. The following corollary then summarizes all the symmetric group and graph-theoretic characterizations of controllability for systems on $\son$.
  
\begin{corollary}
  \label{cor:controllability.summary}
  Consider a bilinear system defined on $\son$ as in
  \cref{eq:system_SOn}, and let $\Gamma$ denote the set of vector
  fields governing the system dynamics. The following are equivalent:
  \begin{enumerate}[font=\normalfont,label={(\arabic*)}]%
    \item The system is controllable on $\son$.
    \item $\tau(\Gamma)$ is a connected graph.
    \item For any minimal subset $\Sigma\subseteq\Gamma$ generating
      $\mathfrak{so}(n)$, $\iota(\Sigma)$ is an $n$-cycle and
      $\tau(\Sigma)$ is a spanning tree of $\tau(\Gamma)$.
  \end{enumerate}
\end{corollary}

In the remainder of this section, we will focus on uncontrollable systems. Recall \cref{thm:controllable.submanifold} that the controllable submanifold for an uncontrollable system on $\son$ is determined by the connected components of its associated graph. Meanwhile, according to \cref{cor:Sn_submanifold}, by applying the method of symmetric groups, the controllable submanifold can also be characterized by the nontrivial orbits of $\iota(\Xi)$ for a minimal subset $\Xi\subseteq\Gamma$ generating $\lie(\Gamma)$. To see that the two methods are equivalent and to extend \cref{thm:spanning.tree} to uncontrollable cases, we first introduce the concept of \emph{spanning forests}, which generalizes the notion of spanning trees to disconnected graphs. Given a (disconnected) graph, its \emph{spanning forest} is a maximal acyclic subgraph, or equivalently, a subgraph consisting of a spanning tree in each connected component of the graph \cite{Bollobas98}. Following this definition, we will show that the minimal subset $\Xi\subseteq\Gamma$ in \cref{cor:Sn_submanifold} corresponds to a spanning forest of $\tau(\Gamma)$, so that the controllable submanifold can also be equivalently described by the connected components of the spanning forest. This result is illuminated in the following example.

\begin{example}
  Consider a bilinear system on $\mathrm{SO}(6)$ in the form of \cref{eq:system_SOn} governed by the set of vector fields $\Gamma=\{\Omega_{12}, \Omega_{14}, \Omega_{23}, \Omega_{24},\Omega_{34}, \Omega_{56}\}$.  As shown in \cref{tab:permutation-graph.uncontrollable}, $\iota(\Gamma)$ is
  disconnected with two components, and hence this system is not controllable on ${\rm SO}(6)$. To describe its controllable submanifold, we choose a spanning forest $\tau(\Xi_1)$ of the associated graph $\tau(\Gamma)$ with $\Xi_1=\{\Omega_{14}, \Omega_{24}, \Omega_{34}, \Omega_{56}\}$. Note that the permutation $\iota(\Xi_1)=(14)(24)(34)(56)=(1432)(56)$ has two nontrivial orbits: $\mathcal{O}_1=\{1,2,3,4\}$ and
  $\mathcal{O}_2=\{5,6\}$, each corresponds to a connected component of the graph $\tau(\Gamma)$, or equivalently, a summand in the decomposition of the Lie algebra of the controllable submanifold:
  \[
    \lie(\Gamma)={\rm
      span}\,\{\Omega_{ij}:i,j\in\mathcal{O}_{1}\}\oplus{\rm
      span}\,\{\Omega_{ij}:i,j\in\mathcal{O}_{2}\}.
  \]

  Now suppose we choose a different spanning forest $\tau(\Xi_2)$
  which corresponds to another subset
  $\Xi_2=\{\Omega_{12}, \Omega_{24}, \Omega_{34},
  \Omega_{56}\}\subseteq\Gamma$. Note that the permutation
  $\iota(\Xi_2)=(1243)(56)$ is different from $\iota{}(\Xi_1)$, but
  both have the \emph{same} orbits. The graphs and permutations
  associated with $\Gamma$ and its subsets $\Xi_1$ and $\Xi_2$ are
  also listed in \cref{tab:permutation-graph.uncontrollable}.
\end{example}

In general, for a spanning forest $F$ of $\tau(\Gamma)$, we know by \cref{thm:spanning.tree} that each tree $T_i$ consisting of $n_i$ vertices in $F$ is related to an $n_i$-cycle via $\iota$, which characterizes a summand of the decomposition of $\lie(\Gamma)$. So by applying \cref{thm:spanning.tree} to each (maximal) tree in the forest $F$, we have the following \cref{cor:spanning.tree.uncontrollable}, which describes the relation between the associated graphs and permutations for an uncontrollable bilinear system.

\begin{table}[tbp]
  \centering
  \begin{tabularx}{\textwidth}{lYl}
    \toprule
    Set of control vector fields
      & Graph
        & \makecell[l]{Permutation in $S_6$ and \\
          its nontrivial orbits} \\
    \midrule
    \(\Gamma=\Bigl\{
      \begin{array}{l}
        \Omega_{12}, \Omega_{14},\Omega_{23}, \\
        \Omega_{24}, \Omega_{34},\Omega_{56}
      \end{array}\Bigr\}\)
      & \begin{tikzpicture}[baseline={-1ex},scale=1.0,thick]
          \node[left] at (180:\R) {$v_1$};
          \node[above left] at (120:\R) {$v_2$};
          \node[above right] at (60:\R) {$v_3$};
          \node[right] at (0:\R) {$v_4$};
          \node[below right] at (300:\R) {$v_5$};
          \node[below left] at (240:\R) {$v_6$};
          
          \draw (180:\R) -- (120:\R);
          \draw (180:\R) node[vertex]{} -- (0:\R);
          \draw (120:\R) -- (60:\R);
          \draw (120:\R) node[vertex]{} -- (0:\R);
          \draw (0:\R) node[vertex]{} -- (60:\R) node[vertex]{};
          \draw (240:\R) node[vertex]{} -- (300:\R) node[vertex]{};
        \end{tikzpicture}
        & \makecell[l]{
          \(\begin{aligned}
            \iota(\Gamma) &=(12)(14)(23)(24)(34)(56) \\
                          &=(14)(56)
            \end{aligned}\) \\
          \gape[t]{$\mbox{Orbits}=\{1,4\}, \{5,6\}$}} \\
    \midrule
    $\Xi_1=\{\Omega_{14},\Omega_{24},\Omega_{34},\Omega_{56}\}$
      & \begin{tikzpicture}[baseline={-1ex},scale=1.0,thick]
          \node[left] at (180:\R) {$v_1$};
          \node[above left] at (120:\R) {$v_2$};
          \node[above right] at (60:\R) {$v_3$};
          \node[right] at (0:\R) {$v_4$};
          \node[below right] at (300:\R) {$v_5$};
          \node[below left] at (240:\R) {$v_6$};
          
          \draw (180:\R) node[vertex]{} -- (0:\R);
          \draw (60:\R) node[vertex]{} -- (0:\R);
          \draw (120:\R) node[vertex]{} -- (0:\R) node[vertex]{};
          \draw (240:\R) node[vertex]{} -- (300:\R) node[vertex]{};
        \end{tikzpicture}
        & \makecell[l]{
          \(\begin{aligned}
            \iota(\Xi_1) &=(14)(24)(34)(56) \\
                         &=(1432)(56)
          \end{aligned}\) \\
          \gape[t]{$\mbox{Orbits}=\{1,2,3,4\}, \{5,6\}$}} \\
    \midrule
    $\Xi_2=\{\Omega_{12},\Omega_{24},\Omega_{34},\Omega_{56}\}$
      & \begin{tikzpicture}[baseline={-1ex},scale=1.0,thick]
            \node[left] at (180:\R) {$v_1$};
            \node[above left] at (120:\R) {$v_2$};
            \node[above right] at (60:\R) {$v_3$};
            \node[right] at (0:\R) {$v_4$};
            \node[below right] at (300:\R) {$v_5$};
            \node[below left] at (240:\R) {$v_6$};

            \draw (180:\R) node[vertex]{} -- (120:\R);
            \draw (120:\R) node[vertex]{} -- (0:\R);
            \draw (0:\R) node[vertex]{} -- (60:\R) node[vertex]{};
            \draw (240:\R) node[vertex]{} -- (300:\R) node[vertex]{};
          \end{tikzpicture}
        & \makecell[l]{
          \(\begin{aligned}
            \iota(\Xi_2) &=(12)(24)(34)(56) \\
                         &=(1243)(56)
          \end{aligned}\) \\
          \gape[t]{$\mbox{Orbits}=\{1,2,3,4\}, \{5,6\}$}} \\
    \bottomrule
  \end{tabularx}
  \caption{A comparison between the symmetric group method and the
    graph-theoretic method for an uncontrollable system on
    $\mathrm{SO}(6)$. Both graphs associated with subsets $\Xi_1$ and
    $\Xi_2$ are \emph{spanning forest} of the associated graph of
    $\Gamma$.}
  \label{tab:permutation-graph.uncontrollable}
\end{table}

\begin{corollary}
  \label{cor:spanning.tree.uncontrollable}
  Given an uncontrollable bilinear system defined on $\son$ in the form of \cref{eq:system_SOn} governed by the set of vector fields $\Gamma$. Let $F$ be a spanning forest of $\tau(\Gamma)$ and if we denote $\Xi=\tau^{-1}(F)$, then $\Xi$ is a minimal subset of $\Gamma$ with the same generating Lie algebra and the controllable submanifold of the system is determined by the nontrivial orbits of $\iota(\Xi)$.
\end{corollary}

\begin{proof}
  For a spanning forest $F$ of $\tau(\Gamma)$, let $T_1, \ldots, T_l$
  be (maximal) trees in $F$ s.t.\ $F=T_1\sqcup\cdots\sqcup{} T_l$,
  where $T_i=(V_i, E_i)$ with $|V_i|=n_i$, $|E_i|=n_i-1$. By
  \cref{thm:spanning.tree}, each $T_i$ is related to an $n_i$-cycle
  $\iota(\Xi_i)\in{}S_{n_i}$ for $\Xi_i=\tau^{-1}(T_i)$, and the orbit
  of $\iota(\Xi_i)$ determines the Lie (sub)algebra
  $\mathfrak{g}_i:=\lie(\Xi_i)$. Therefore, distinct orbits of
  $\iota(\Xi)$ consist of the orbits of each
  $\iota(\Xi_1), \ldots, \iota(\Xi_l)$, which determines the Lie
  algebra generated by $\Gamma$:
  $\lie(\Gamma)=\lie(\Xi)=\mathfrak{g}_i\oplus\cdots\oplus{}
  \mathfrak{g}_l$. Since the controllable submanifold is determined by
  the Lie subalgebra $\lie(\Gamma)=\lie(\Xi)$, we conclude that it is
  also determined by distinct orbits of $\iota(\Xi)$, for
  $\Xi=\tau^{-1}(F)$.
\end{proof}


\section{Combinatorics-Based Controllability Analysis via Lie Algebra Decompositions} 
\label{sec:non-standard}
 
Utilizing the algebraic structure of $\mathfrak{so}(n)$, we have developed combinatorial methods that identified vector fields in the standard basis of $\mathfrak{so}(n)$, as well as vector fields generating structured Lie algebras, e.g., the multi-agent system described in \cref{ex:formation.control}, with transpositions in $S_n$ and edges of $n$-vertices graphs. It was also shown that such identifications lead to an equivalence between the two methods for analyzing controllability of systems on $\son$ as defined in \cref{eq:system_SOn}.

However, in many cases, the system Lie algebra may be too complicated to associate each of its elements to a permutation or a graph edge, so that the combinatorial methods cannot be directly applied. This dilemma can be resolved through the decomposition of the Lie algebra into components with simpler algebraic structures such that the combinatorial methods can be applied to each component. This idea allows us to generalize the combinatorial framework to bilinear systems defined on boarder classes of Lie groups. To this end, we adopt techniques in representation theory, including the Cartan and non-intertwining decomposition. Some basics of representation theory can be found in \cref{appd:representation}.


\subsection{Cartan Decomposition in Symmetric Group Method} \label{sec:Cartan}

The Cartan decomposition, named after the influential French mathematician \'{E}lie Cartan, provides a major tool for understanding the algebraic structures of semisimple Lie groups and Lie algebras. Its generalized form, the root space decomposition, decomposes a Lie algebra into a direct sum of vector subspaces, called the root spaces, as introduced in \cref{appd:representation}. However, each root space is not necessarily a Lie subalgebra, i.e., Lie bracket
operations may not be closed in the root spaces. This nature of the
Cartan (root space) decomposition then disables the use of the graph-theoretic method since it violates the ``triangle rule'' shown in \cref{lem:lie-graph} (ii). 
As a result, here we pursue and generalize the symmetric group method to analyze controllability of systems with its vector fields living in the root spaces of semisimple Lie algebras.

In representation theory, the Lie algebra
$\mathfrak{sl}(3,\mathbb{C})$, which consists of $3\times{} 3$ complex
matrices with vanishing trace, serves as a primary example to illustrate the Cartan decomposition of semisimple Lie
algebras. Therefore, to illustrate our idea, we consider the driftless bilinear system evolving on the Lie group ${\rm SL}(3,\mathbb{C})$ consisting of
$3\times 3$ complex matrices with determinant~$1$, given by 
\begin{equation}
  \label{eq:sl3}
  \dot{Z}(t)=\Bigl(\sum_{j=1}^m{} u_j(t)B_j\Bigr)Z(t),
\end{equation}
where the state $Z(t)\in {\rm SL}(3,\mathbb{C})$, the control vector fields $B_j\in\Gamma\subseteq\mathcal{B}'':=\{H_k,X_l,Y_l: k=1,2; l=1,2,3\}$, the basis of
$\mathfrak{sl}(3,\mathbb{C})$ with
\begingroup
\allowdisplaybreaks
\begin{align*}
  H_1 &=\begin{bmatrix}
         1 & 0 & 0 \\ 0 & -1 & 0 \\ 0 & 0 & 0 
       \end{bmatrix}, &
  H_2 &=\begin{bmatrix}
         0 & 0 & 0 \\ 0 & 1 & 0 \\ 0 & 0 & -1
      \end{bmatrix}, & \\
  X_1 &=\begin{bmatrix}
         0 & 1 & 0 \\ 0 & 0 & 0 \\ 0 & 0 & 0
        \end{bmatrix}, &
  X_2 &=\begin{bmatrix}
         0 & 0 & 0 \\ 0 & 0 & 1 \\ 0 & 0 & 0
       \end{bmatrix}, &
  X_3 &=\begin{bmatrix}
         0 & 0 & 1 \\ 0 & 0 & 0 \\ 0 & 0 & 0 
       \end{bmatrix}, \\
  Y_1 &=\begin{bmatrix}
         0 & 0 & 0 \\ 1 & 0 & 0 \\ 0 & 0 & 0
       \end{bmatrix}, &
  Y_2 &=\begin{bmatrix}
         0 & 0 & 0 \\ 0 & 0 & 0 \\ 0 & 1 & 0
       \end{bmatrix}, &
  Y_3 &=\begin{bmatrix}
         0 & 0 & 0 \\ 0 & 0 & 0 \\ 1 & 0 & 0
       \end{bmatrix},
\end{align*}
\endgroup
and the control inputs $u_j(t)\in\mathbb{C}$.

One can easily check that the two Lie subalgebras $\mathfrak{k}_1=\lie\{H_1, X_1,Y_1\}$ and $\mathfrak{k}_2=\lie\{H_2,X_2,Y_2\}$, when considered as Lie algebras over $\mathbb{R}$, are isomorphic to $\mathfrak{so}(3)$. As discussed in Section~\ref{sec:Sn}, controllability of systems on $\mathrm{SO}(3)$ can be characterized by permutation cycles in $S_3$. This suggests that we can characterize controllability of
systems on $\mathrm{SL}(3,\mathbb{C})$ by two copies of $S_3$. Formally, we want to establish a map $\iota:\mathcal{P}(\mathcal{B''})\rightarrow S_3\oplus S_3$, where $\oplus$ denotes the direct sum of groups, so that non-vanishing Lie brackets correspond to cycles with increased length. In this case, we define an element $\sigma=(\sigma_1,\sigma_2)$ in $S_3\oplus S_3$ to be a cycle if both $\sigma_1$ and $\sigma_2$ are cycles in $S_3$, and the length of $\sigma$ is defined to be the sum of the length of $\sigma_1$ and $\sigma_2$. Here is one possible definition of $\iota$:
\begin{align*}
  H_1 &\mapsto (e,e), & H_2 &\mapsto (e,e), & \\
  X_1 &\mapsto ((12),e), & X_2 &\mapsto (e,(12)), & X_3
  &\mapsto ((12),(12)), \\
  Y_1 &\mapsto ((23),e), & Y_2 &\mapsto (e,(23)), & Y_3
  &\mapsto ((23),(23)),
\end{align*}
where $e$ denotes the identity of $S_3$. Following this definition of $\iota$, we can check that if
$B_1,B_2\in\mathcal{B}''$ satisfy $[B_1,B_2]\neq0$, then the length of $\iota([B_1,B_2])$ is greater than or equal to the length of both $\iota(B_1)$ and $\iota(B_2)$. Moreover, if neither $B_1$ nor $B_2$ is equal to $H_1$ or $H_2$, then the length of $\iota([B_1,B_2])$ is strictly greater than the length of both
$\iota(B_1)$ and $\iota(B_2)$. This relation between Lie brackets of elements in $\mathcal{B}''$ and length of cycles in $S_3\oplus S_3$ allows us to draw the following conclusion:

\begin{proposition}
  \label{prop:sl3c}
  The system in \cref{eq:sl3} is controllable on
  ${\rm SL}(3,\mathbb{C})$ if and only if there exists a subset
  $\Sigma$ of $\Gamma=\{B_1,\dots,B_m\}$ such that $\iota(\Sigma)$ is
  a 6-cycle in $S_3\oplus S_3$.
\end{proposition}

From the perspective of representation theory, the basis $\mathcal{B}''$ induces the Cartan decomposition of the Lie algebra $\mathfrak{sl}(3,\mathbb{C})$, in which the $2$-dimensional Cartan subalgebra is spanned by $H_1$ and $H_2$. Moreover, the Weyl group of $\mathfrak{sl}(3,\mathbb{C})$ is $S_3$. The above facts provide another explanation for requiring two copies of $S_3$ in the characterization of controllability for systems on ${\rm SL}(3,\mathbb{C})$ governed by vector fields in $\mathcal{B}''$. Notice that the concepts of Cartan subalgebras and Weyl groups are well-defined for all semisimple Lie algebras, not only for ${\rm SL}(3,\mathbb{C})$. Also, Weyl groups are all finite groups and thus subgroups of some symmetric groups. As a result, it is possible to extend the symmetric-group characterization of controllability to systems defined on general semisimple Lie groups. To be more specific, consider the bilinear system defined on a semisimple Lie group $G$ of the form,
\begin{equation}
  \label{eq:semisimple}
  \dot X=\Bigl(\sum_{i=1}^mu_iB_i\Bigr)X,\quad X(0)=I,
\end{equation}
where $B_i$ are elements in the Lie algebra $\mathfrak{g}$ of $G$. Moreover, let $\mathfrak{g}=\mathfrak{h}\oplus\mathfrak{k}$ be the Cartan decomposition of $\mathfrak{g}$ with $\mathfrak{h}$ being the Cartan subalgebra and $W$ be the Weyl group of $\mathfrak{g}$. We further assume that 
$B_i\in\mathfrak{h}$ or $B_i\in\mathfrak{k}$ for every $i=1,\dots,m$, then 
the above discussion leads to the following conjecture for systems defined on semisimple Lie groups.

\begin{conjecture}
  \label{conj:cartan.decomp.controllable}
  The system in \cref{eq:semisimple} is controllable on $G$ if and only if there exits $\Sigma\subseteq\Gamma$ such that $\iota(\Sigma)$ is a cycle of maximal length in $W^h$, where $\Gamma=\{B_1,\dots,B_m\}$ is the set of control vector fields, $h=\dim\mathfrak{h}$, and $W^h$ denotes the direct sum of $h$ copies of $W$.
\end{conjecture}

Recall that the central idea of the symmetric group approach to controllability analysis is to map elements with non-vanishing Lie brackets to cycles with increased length. However, all elements in the Cartan subalgebra have vanishing Lie brackets. The intuition behind the above conjecture comes from the need of appropriately representing these elements using permutations by mapping elements in different root spaces to permutation cycles in different components of the direct sum of $h$ copies of symmetric groups, where $h$ denotes the dimension of the Cartan subalgebra.
Moreover, because the interaction between elements in and outside the Cartan subalgebra is characterized by the Weyl group, which is a subgroup of a symmetric group, the symmetric group method applies directly.

\subsection{Non-Intertwining Decomposition in Graph-Theoretic Method}

In the case that the Lie algebra generated by drift and control vector fields of a bilinear system can be decomposed into components that are Lie subalgebras, we will see that the graph-theoretic method applies more naturally for controllability analysis. One decomposition of this type is the \emph{non-intertwining decomposition}, through which a Lie algebra is decomposed into a direct sum of Lie subalgebras so that elements from different Lie subalgebras have vanishing Lie brackets. The non-intertwining decomposition generalizes the notion of block diagonalization for matrices.

\begin{definition}
  \label{def:non-intertwining.decomp}
  For a given Lie algebra $\mathfrak{g}$, we call a decomposition $\mathfrak{g}=\mathfrak{g}_1\oplus\cdots\oplus \mathfrak{g}_m$ \emph{non-intertwining} if $[\mathfrak{g}_i, \mathfrak{g}_j]=0$ for
  any Lie subalgebras $\mathfrak{g}_i, \mathfrak{g}_j$, $1 \leqslant{}i\neq{}j \leqslant{}m$.
\end{definition}

For example, every reductive Lie algebra admits a non-intertwining decomposition, and many familiar Lie algebras are reductive, such as the algebra of $n\times n$ complex matrices $\mathfrak{gl}(n,\mathbb{C})$ and the algebra of $n\times n$ skew-symmetric complex matrices $\mathfrak{so}(n,\mathbb{C})$ \cite{Knapp2002lie}. If a Lie algebra admits a non-intertwining decomposition, then we will be able to associate each of its components with a graph. The subsequent question is whether graph representation developed in Section \ref{sec:graph} remains valid to characterize controllability. The answer to this question can be illustrated by a system defined on ${\rm SO}(4)$ whose Lie algebra $\mathfrak{so}(4)$ can be decomposed into a direct sum of two non-intertwining copies of $\mathfrak{so}(3)$, as shown in the following example.

\begin{example}
  \label{ex:spin.reps}
  Let $\mathcal{B}'=\{A_1,A_2,A_3,B_1,B_2,B_3\}$ be a non-standard
  basis of $\mathfrak{so}(4)$, where
  \begin{equation}
    \label{eq:spin.reps.so4.basis}
    \begin{aligned}
      A_1 &=\frac{\Omega_{23}+\Omega_{14}}{2}, &
      A_2 &=\frac{\Omega_{13}-\Omega_{24}}{2}, &
      A_3 &=\frac{\Omega_{12}+\Omega_{34}}{2}, \\%
      B_1 &=\frac{\Omega_{13}+\Omega_{24}}{2}, &
      B_2 &=\frac{\Omega_{14}-\Omega_{23}}{2}, &
      B_3 &=\frac{\Omega_{12}-\Omega_{34}}{2}.
    \end{aligned}
  \end{equation}
  The Lie brackets of the elements in $\mathcal{B}'$ satisfy $[A_i,A_j]=A_k$, $[B_i,B_j]=B_k$ for any ordered 3-tuple $(i,j,k)=(1,2,3)$, $(2,3,1)$ or $(3,1,2)$, and $[A_i,B_j]=0$ for any $1 \leqslant{} i,j \leqslant{} 3$. As a result, $\mathfrak{so}(4)$ admits a non-intertwining decomposition as
  $\mathfrak{so}(4)=\lie\{A_1,A_2,A_3\}\oplus\lie\{B_1,B_2,B_3\}$.

  We note that the Lie bracket relations among elements in $\{A_1,A_2,A_3\}$, as well as $\{B_1,B_2,B_3\}$, are the same as the Lie bracket relations among elements in the standard basis of $\mathfrak{so}(3)$. In other words, both $\lie\{A_1,A_2,A_3\}$ and $\lie\{B_1,B_2,B_3\}$ are isomorphic to $\mathfrak{so}(3)$, so $K_3$ becomes the suitable graph representation for each set. Moreover, because $[A_i,B_j]=0$ for any $i,j=1,2,3$, the graph representation for the non-standard basis $\mathcal{B}'=\{A_1,A_2,A_3\}\sqcup\{B_1,B_2,B_3\}$ is a \emph{disjoint union} of two copies of $K_3$, as shown in \cref{fig:spin.reps}, instead of the complete graph $K_4$ associated with the standard basis of $\mathfrak{so}(4)$.

  \begin{figure}[htbp]
    \centering
    \medskip
    \begin{tikzpicture}[semithick]
      \node[above] at (90:\R) {$v_1$};
      \node[below right] at (330:\R) {$v_2$};
      \node[below left] at (210:\R) {$v_3$};
      
      \node at (20:1.2*\R) {$\tau'(A_1)$};
      \node at (160:1.2*\R) {$\tau'(A_3)$};
      \node at (270:0.9*\R) {$\tau'(A_2)$};
      
      \draw (90:\R) -- (210:\R);
      \draw (210:\R) node[vertex]{} -- (330:\R);
      \draw (330:\R) node[vertex]{} -- (90:\R) node[vertex]{};

      \node[above, xshift=4.5*\R] at (90:\R) {$w_1$};
      \node[below right, xshift=4.5*\R] at (330:\R) {$w_2$};
      \node[below left, xshift=4.5*\R] at (210:\R) {$w_3$};

      \node [xshift=4.5*\R] at (20:1.2*\R) {$\tau'(B_1)$};
      \node [xshift=4.5*\R] at (160:1.2*\R) {$\tau'(B_3)$};
      \node [xshift=4.5*\R] at (270:0.9*\R) {$\tau'(B_2)$};
      
      \draw[xshift=4.5*\R] (90:\R) -- (210:\R);
      \draw[xshift=4.5*\R] (210:\R) node[vertex]{} -- (330:\R);
      \draw[xshift=4.5*\R] (330:\R) node[vertex]{} -- (90:\R) node[vertex]{};
    \end{tikzpicture}
    \caption{The graphs associated with the sets $\{A_1,A_2,A_3\}$ and
      $\{B_1,B_2,B_3\}$ in \cref{ex:spin.reps}.}
    \label{fig:spin.reps}
  \end{figure}
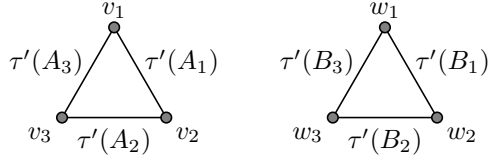

  This example illuminates how the graph representation of controllability developed in Section \ref{sec:graph} can be extended to the bilinear system governed by vector fields generating a non-intertwining Lie algebra, after modifying the definition of $\tau$ in \cref{eq:graph-map} accordingly.
  %

  \begin{proposition}
    \label{prop:controllability.two.K3s}
    Consider a bilinear system on $\mathrm{SO}(4)$ governed by the vector fields in $\mathcal{B}'$, given by
    \begin{equation}
      \label{eq:so4_non_standard}
      \dot{X}(t)=\Bigl(\sum_{i=1}^mu_iC_i\Bigr)X(t), \quad X(0)=I,
    \end{equation}
    with $\Gamma=\{C_1,\dots,C_m\}\subseteq\mathcal{B}'$. Given a graph map $\tau':\mathcal{P}(\mathcal{B}')\to\mathcal{G}'$, where $\mathcal{G}'$ denotes the collection of subgraphs of
    $K_3\sqcup K_3$, satisfying
    \[
      \tau'(A_i)=v_iv_{i+1} \quad \text{and} \quad
      \tau'(B_i)=w_iw_{i+1},
    \]
    with the index taken modulo $3$, the system in \cref{eq:so4_non_standard} is controllable if and only if $\overline{\tau'(\Gamma)}=K_3\sqcup K_3$, or equivalently, if and only if each component of $\tau'(\Gamma)$ is connected in $K_3$.
  \end{proposition}

  \begin{proof}
    The above result becomes obvious once we verify the following properties of $\tau'$ (c.f.\ \cref{lem:lie-graph}), which are straightforward.
    \begin{enumerate}[font=\normalfont,label={(\arabic*)}]%
      \item $\tau'(\mathcal{B}')=K_3\sqcup{}K_3$;
      \item For distinct $C_1,C_2\in\mathcal{B}'$, their Lie bracket $[C_1,C_2]\neq{}0$ if and only if the two edges $\tau'(C_1)$ and $\tau'(C_2)$ have a common vertex;
      \item The edges $\tau'(C_1),\tau'(C_2)$ and $\tau'([C_1,C_2])$ form a triangle if $[C_1,C_2]\neq{}0$, or equivalently,
        \[
          \tau'(\{C_1,C_2,[C_1,C_2]\})=K_3,
        \]
        for any $C_1,C_2\in\mathcal{B}'$ such that $[C_1,C_2]\neq{}0$.
    \end{enumerate}
  \end{proof}
  
  In addition, recall from \cref{cor:control_number} that three control inputs are enough to have a controllable driftless system on $\mathrm{SO}(4)$ governed by the vector fields in the standard basis; or equivalently, three edges can form a connected graph with four vertices. However, for systems in the form of \cref{eq:so4_non_standard}, they require at least four control inputs to be controllable on $\mathrm{SO}(4)$. From the graph aspect, this is because both components of $\tau(\Gamma)$ require at least two edges to be connected.
\end{example}

\Cref{ex:spin.reps} further illustrates that for bilinear systems evolving on $\son$ governed by non-standard basis vector fields, i.e., vector fields that are not in the form of standard basis elements, in $\mathfrak{so}(n)$, 
controllability may not be characterized by using one complete graph $K_n$. Taking the system in \cref{eq:so4_non_standard} as an example, because the Lie algebra of its state-space can be decomposed into a direct sum of two non-intertwining components, its graph representation also requires two components. This finding elucidates that the number of components of the graph associated with a bilinear system is determined by the number of summands in the non-intertwining decomposition of the underlying Lie algebra of the system.

\begin{theorem}
  \label{thm:direct.sum.lie.algebra.disjoint.union.graphs}
  Given a bilinear system
  \begin{equation}
    \label{eq:general.non.intertwining.system}
    \dot{X}(t)=\Bigl(\sum_{i=1}^m\sum_{j=1}^{n_i} u_{ij}B_{ij}\Bigr)X(t), \quad
    X(0)=I,
  \end{equation}
  defined on a Lie group $G$ whose Lie algebra $\mathfrak{g}$ admits a non-intertwining decomposition as $\mathfrak{g}=\mathfrak{g}_1\oplus\cdots\oplus\mathfrak{g}_m$, where $B_{ij}\in\mathcal{B}_i$ and $\mathcal{B}_i$ is a basis of $\mathfrak{g}_i$ for each $i$. Suppose each $\mathcal{B}_i$ is associated with a connected graph $G_i$ such that a subset $\Sigma_i\subseteq\mathcal{B}_i$ generates $\mathfrak{g}_i$ if and only if its associated graph $\tau(\Sigma_i)$ is a connected subgraph of $G_i$, then the system in \cref{eq:general.non.intertwining.system} is controllable on $G$ if and only if $\tau(\Gamma_i)$ is connected for every $i=1,\dots,m$, where $\Gamma_i=\{B_{ij}:j=1,\dots, n_i\}$.
  %
\end{theorem}

\begin{proof}
  By the assumption, $\tau(\Gamma_i)$ is connected if and only if
  $\lie(\Gamma_i)=\mathfrak{g}_i$ for each $i=1,\dots,m$. Together
  with the non-intertwining property between each pair of
  $\mathfrak{g}_i$ and $\mathfrak{g}_j$, the connectivity of
  $\tau(\Gamma_i)$ for all $i$ is equivalent to
  \[
    \lie(\Gamma) =\bigoplus_{i=1}^m\lie(\Gamma_i)
    =\bigoplus_{i=1}^m\mathfrak{g}_i =\mathfrak{g},
  \]
  where $\Gamma=\bigcup_{i=1}^m\Gamma_i$. The proof is then concluded
  by applying the LARC.
  %
\end{proof}

\begin{remark}[Symmetric Group Method for Systems Governed by Non\hyp{}intertwining Lie Algebras]
  We find it worthwhile to mention that the symmetric group method also applies to bilinear systems with their underlying Lie algebras admitting a non-intertwining decomposition, through a properly defined $\iota$. For instance, in \cref{ex:spin.reps}, since both $\{A_i\}$ and $\{B_i\}$ in \cref{eq:spin.reps.so4.basis} are isomorphic to the standard basis in $\mathfrak{so}(3)$, the symmetric group method extends to the systems in \cref{eq:so4_non_standard} as well, by associating each component in the decomposition to a copy of $S_3$ and defining $\iota(A_i,B_j)=\bigl((i,i+1), (j,j+1)\bigr)$, with the index taken modulo $3$. Consequently, the system in \cref{eq:so4_non_standard} is controllable if and only if $\iota$ relates $\Gamma$ to two
  disjoint $3$-cycles in $S_3\oplus S_3$.
\end{remark}

\section{Summary}

In this paper, we develop a combinatorics-based framework to characterize controllability of bilinear systems evolving on Lie groups, in which Lie bracket operations of vector fields are represented by operations on permutations in a symmetric group and edges in a graph. Through such representations, we obtain the tractable and transparent combinatorial characterizations of controllability in terms of permutation cycles and graph connectivity. This framework is established by first considering bilinear systems on $\son$, and we show that, in this case, the permutation and graph representations are equivalent. Then, by exploiting techniques in representation theory, we extend our investigation into a more general category of bilinear systems via proper decompositions of the underlying Lie algebras of the systems. In particular, we illustrate the application of the developed combinatorial methods to bilinear systems whose underlying Lie algebras admit the Cartan or non-intertwining decomposition. The presented methodology not only provides an alternative to the LARC, but also advances geometric control theory 
by integrating it with techniques in combinatorics and representation theory. As a final remark, compared to known graph-theoretic methods mostly developed for networked or multi-agent systems, our framework proposes novel applications of graphs to the study of bilinear control systems.

\appendix

\section{Symmetric Groups and Permutations} \label{appd:Sn}

In this appendix, we give a brief review of the symmetric group
theory. For a thorough discussion on symmetric groups, the reader can
refer to any standard algebra textbook, for example~\cite{Lang02}. Let
$X_n$ be a finite set of $n$ elements, and without loss of generality,
we may assume $X_n=\{1,\cdots,n\}$. A \emph{permutation} $\sigma$ of
$X_n$ is a bijection from $X_n$ onto itself, and is denoted by
\[
  \sigma=\begin{pmatrix}
    1 & 2 & \cdots & n \\
    i_1 & i_2 & \cdots & i_n
  \end{pmatrix}
\] 
if $\sigma(1)=i_1$, \dots, $\sigma(n)=i_n$ for distinct
$i_1,\ldots,i_n\in{}X_n$. A permutation that switches only two
elements is called a \emph{transposition}, and is denoted by
$\sigma=(i_1i_2)$ if $i_1\neq i_2$ and $\sigma$ fixes all other
indices except for $\sigma(i_1)=i_2$ and $\sigma(i_2)=i_1$. More
generally, an \emph{$r$-cycle} denoted by $\sigma=(i_1i_2\cdots{}i_r)$
is a permutation that satisfies $\sigma(i_1)=i_2$,
$\sigma(i_{2})=i_3$, \ldots, $\sigma(i_r)=i_1$ and fixes all other
indices.  It can be shown that any permutation can be decomposed
uniquely into disjoint cycles (cycles that have no common indices).
For example, when $n=4$, the permutation \(\bigl(\begin{smallmatrix}
  1 & 2 & 3 & 4\\
  2 & 3 & 4 & 1
\end{smallmatrix}\bigr)\)
can be represented by a single $4$-cycle $(1234)$; while the
permutation \(\bigl(\begin{smallmatrix}
  1 & 2 & 3 & 4 \\
  3 & 4 & 1 & 2
\end{smallmatrix}\bigr)\)
is the composition of two transpositions ($2$-cycles):
$(13)(24)$. Given a permutation $\sigma$ of $X_n$ and an integer
$i$, $1\leqslant{}i\leqslant{}n$, the \emph{orbit} of $i$ is formed
under the cyclic group generated by $\sigma$. So for
$\sigma=(1234)$, the orbit of $2$ is
$\{\sigma^i(2):i\in\mathbb{N}\} =\{2, \sigma(2), \sigma^2(2),
\sigma^3(2)\} =\{1,2,3,4\};$ and for $\sigma=(13)(24)$, the orbit of
$2$ is $\{\sigma^i(2):i\in\mathbb{N}\} =\{2,\sigma(2)\} =\{2,4\}$.
The \emph{symmetric group} $S_n$ is defined as the group of
permutations on $X_n$, with its group operation being the
composition of bijections.

\section{Basics of Representation Theory} \label{appd:representation}

Representation theory is a branch of algebra which studies structure
theory by representing elements in an algebraic object, such as a
group, a module, or an algebra, using linear transformations of vector
spaces. In this appendix, we will review some basic concepts and
results in the representation theory of Lie algebras that are used in
this paper. Detailed discussions of Lie representation theory can be
found in \cite{fulton1991,Knapp2002lie}.


To study the algebraic structure of a Lie algebra, let us introduce
some related definitions.
\begin{definition}
  \label{def:semisimple.appendix}
  \mbox{}
  \begin{itemize}
  \item A Lie algebra $\mathfrak{g}$ is said to be \emph{abelian} if
    \[
      [\mathfrak{g},\mathfrak{g}]:={\rm
        span}\,\{[X,Y]:X,Y\in\mathfrak{g}\}=0.
    \]
  \item A subspace $\mathfrak{h}$ of $\mathfrak{g}$ is a \emph{Lie
      subalgebra} of $\mathfrak{g}$ if
    $[\mathfrak{h},\mathfrak{h}]\subseteq\mathfrak{h}$. In other
    words, $\mathfrak{h}$ is a Lie algebra itself w.r.t.\
    $[\cdot,\cdot]$.
  \item A Lie subalgebra $\mathfrak{h} \leqslant{}\mathfrak{g}$ is an
    \emph{ideal} in $\mathfrak{g}$ if
    $[\mathfrak{h},\mathfrak{g}]\subseteq\mathfrak{h}$.
  \item The Lie algebra $\mathfrak{g}$ is said to be \emph{simple} if
    it is nonabelian and has no proper nonzero ideals, and
    \emph{semisimple} if it has no nonzero abelian ideals.
  \end{itemize}
\end{definition}

It can be shown that every semisimple Lie algebra $\mathfrak{g}$ can
be decomposed into a direct sum of simple Lie algebras which are
ideals in $\mathfrak{g}$. Moreover, this decomposition is unique, and
the only ideals of $\mathfrak{g}$ are the direct sums of some of these
simple Lie algebras. For example, each special orthogonal Lie algebra
$\mathfrak{so}(n)=\{\Omega\in\mathbb{R}^{n\times
  n}:\Omega+\Omega^{\tr}=0\}$, as we use extensively in this paper, is
simple except for $n=4$, while $\mathfrak{so}(4)$ is semisimple but
not simple: as shown in \cref{ex:spin.reps},
$\mathfrak{so}(4)=\mathfrak{so}(3)\oplus\mathfrak{so}(3)$.

The study of algebraic structures of semisimple Lie algebras plays a
central role in representation theory. One of the most dominant
results is the Cartan decomposition that traces back to the work of
\'{E}lie Cartan and Wilhelm Killing in the 1880s, which generalizes
the notion of singular value decomposition for matrices. Given a
semisimple Lie algebra $\mathfrak{g}$, its \emph{Cartan subalgebra}
$\mathfrak{h}$ is a maximal abelian subalgebra of $\mathfrak{g}$ such
that ${\rm ad}_H$ is diagonalizable for all $H\in\mathfrak{h}$, where
${\rm ad}_XY=[X,Y]$ for all $X,Y\in\mathfrak{g}$. Moreover, the
dimension of $\mathfrak{h}$ is called the \emph{rank} of
$\mathfrak{g}$. Let $\mathfrak{h}^{\ast}$ denote the dual space of
$\mathfrak{h}$, i.e., the space of linear functionals on
$\mathfrak{h}$, then a nonzero element $\alpha\in\mathfrak{h}$ is
called a \emph{root} of $\mathfrak{g}$ if there exists some
$X\in\mathfrak{g}$ such that ${\rm ad}_HX=\alpha(H)X$ for all
$H\in\mathfrak{h}^{\ast}$, and
$\mathfrak{g}_{\alpha}:=\{X\in\mathfrak{g}: {\rm ad}_HX=\alpha(H)X,
\forall H\in\mathfrak{h}\}$ is a vector space called the \emph{root
  space} of $\mathfrak{g}$, which can be shown to be
one-dimensional. Let $R$ denote the set of roots of $\mathfrak{g}$,
then $R$ is finite and spans $\mathfrak{h}^*$. With the above
notations, the \emph{root space decomposition}, which generalizes the
classical \emph{Cartan decomposition}, is defined as
\[
  \mathfrak{g}=\mathfrak{h}\oplus\Bigl(\bigoplus_{\alpha\in
    R}\mathfrak{g}_{\alpha}\Bigr).
\]
A major tool to study the properties of $R$ is the Weyl group, which
is defined as follows: Let $\alpha\in R$ be a root and
$s_{\alpha}:\mathfrak{h}^*\rightarrow\mathfrak{h}^*$ denote the
reflection about the hyperplane in $\mathfrak{h}^*$ orthogonal to
$\alpha$, i.e.,
$s_\alpha(\beta)
=\beta-\frac{2\langle\beta,\alpha\rangle}{\langle\alpha,\alpha\rangle}\alpha$
for all $\beta\in\mathfrak{h}^*$, where $\langle\cdot,\cdot\rangle$ is
an inner product on $\mathfrak{h}$, then the \emph{Weyl group} $W$ of
$R$ is the subgroup of the orthogonal group ${\rm O}(\mathfrak{h}^*)$
of $\mathfrak{h}^*$ generated by all $s_\alpha$ for $\alpha\in R$. It
can be shown that $W$ is a finite group and hence a subgroup of a
symmetric group by Cayley's theorem.

\bibliographystyle{siamplain}
\bibliography{SOn}
\end{document}